\documentclass[reqno,tbtags,a4paper,12pt]{amsart}

\usepackage[in]{fullpage}

\usepackage{amsmath,amssymb,amsthm,amsfonts,amscd}

\renewcommand{\le}{\leqslant}
\renewcommand{\ge}{\geqslant}

\newcommand{\sgn}{\mathop{\rm sgn}\nolimits}

\newcommand{\T}{\mathcal{T}}

\newcommand{\CP}{\mathbb{C}\mathrm{P}}
\newcommand{\RP}{\mathbb{R}\mathrm{P}}

\renewcommand{\P}{\mathcal{P}}

\newcommand{\N}{\mathcal{N}}
\newcommand{\Z}{\mathbb{Z}}
\newcommand{\R}{\mathbb{R}}
\newcommand{\bC}{\mathbb{C}}
\newcommand{\bp}{\mathbf{p}}
\newcommand{\bt}{\mathbf{t}}
\newcommand{\bX}{\overline{X}}
\newcommand{\hU}{\widehat{U}}

\newcommand{\Sym}{\mathop{\mathrm{Sym}}\nolimits}
\newcommand{\hrho}{\hat{\rho}}
\newcommand{\conj}{c}

\newtheorem{theorem}{Theorem}[section]
\newtheorem{propos}[theorem]{Proposition}

\theoremstyle{definition}
\newtheorem{constr}[theorem]{Construction}
\newtheorem{ex}[theorem]{Example}
\newtheorem{defin}[theorem]{Definition}
\newtheorem{remark}[theorem]{Remark}

\author{A. A. Gaifullin}

\thanks{The work was partially supported by the Russian Foundation
for Basic Research (grants no.~08-01-00541-а
and~08-01-91958-NNIO-a) and the Russian Programme for the Support
of Leading Scientific Schools (grant no.~NSh-1824.2008.1).}

\title{A minimal triangulation of complex projective plane
admitting a chess colouring of four-dimensional simplices}

\date{}

\address{Moscow State University}

\address{Institute for Information Transmission Problems}

\begin{document}

\maketitle

\begin{abstract}
In this paper we construct and study a new 15-vertex
triangulation~$X$ of the complex projective plane~$\CP^2$. The
automorphism group of~$X$ is isomorphic to~$S_4\times S_3$. We
prove that the triangulation~$X$ is the minimal by the number of
vertices triangulation of~$\CP^2$ admitting a chess colouring of
four-dimensional simplices. We provide explicit parametrizations
for simplices of~$X$ and show that the automorphism group of~$X$
can be realized as a group of isometries of the Fubini--Study
metric. We provide a 33-vertex subdivision~$\bX$ of the
triangulation~$X$ such  that the classical moment
mapping~$\mu:\CP^2\to\Delta^2$ is a simplicial mapping of the
triangulation~$\bX$ onto the barycentric subdivision of the
triangle~$\Delta^2$. We study the relationship of the
triangulation~$X$ with complex crystallographic groups.
\end{abstract}

\section*{Introduction}

The complex projective plane~$\CP^2$ has a 9-vertex
triangulation~$K$  constructed by W.\,K\"uhnel in 1980
(see~\cite{KuBa83},~\cite{KuLa83}). The simplicial complex~$K$ is
a remarkable combinatorial object and has many interesting
properties.

\begin{itemize}
\item The triangulation~$K$ is the minimal by the number of
vertices triangulation of~$\CP^2$. Moreover, $K$ is minimal among
all four-dimensional triangulated manifolds not homeomorphic to
the sphere.
\item The triangulation~$K$ is 3-neighbourly, that is,
any set of 3 vertices of~$K$ spans a two-dimensional simplex.
\item The automorphism group of~$K$ has order~54. Besides, there
is a homeomorphism $|K|\approx\CP^2$ such that the automorphism
group of~$K$ acts by isometries of the Fubini-Study metric.
\item
The simplicial complex~$K$ has an interesting interpretation in
terms of the affine plane over the 3-element field~\cite{BaDa94}.
\item The triangulation~$K$ has deep relationship with theory of
complex crystallographic groups. In particular, there exist a
rectilinear triangulation~$\widehat{K}$ of~$\bC^2$ and a complex
crystallographic group~$\Gamma$ acting on~$\bC^2$ such that
$\bC^2/\Gamma\approx\CP^2$ and $\widehat{K}/\Gamma=K$
(see~\cite{MoYo91},~\cite{ArMa91}; see also
Section~\ref{section_crystal} of the present paper).
\end{itemize}

In the present paper we construct and study a new 15-vertex
triangulation~$X$ of~$\CP^2$, which possesses several interesting
properties. The vertices of~$X$ admit a regular colouring in~$5$
colours. (The colouring is said to be \textit{regular} if any
four-dimensional simplex contains exactly one vertex of each
colour.) The four-dimensional simplices of~$X$ admit a chess
colouring, that is, a colouring in two colours, black and white,
such that any two simplices possessing a common facet are of
different colours. Moreover, $X$ is the minimal by the number of
vertices triangulation of~$\CP^2$ admitting a regular colouring of
vertices and a chess colouring of four-simplices. (It is easy to
prove that a triangulation of a simply-connected manifold admits a
regular colouring of vertices if and only if it admits a chess
colouring of maximal-dimensional simplices.) The study of
triangulations admitting either regular colourings of vertices or
chess colourings of maximal-dimensional simplices is motivated by
the works of I.\,A.\,Dynnikov and
S.\,P.\,Novikov~\cite{NoDy97}--\cite{DyNo03}. In these papers such
triangulations are very important for the discretization of
differential geometric connections and complex analysis (for
details, see section~\ref{section_colour}).

The automorphism group of~$X$, which we denote by~$\Sym(X)$, has
order~$144$ and is isomorphic to the direct product~~$S_4\times
S_3$, where $S_n$ is the permutation group of an $n$-element set.
In section~\ref{section_explicit} we shall give an explicit
formula for the homeomorphism~$f:|X|\to\CP^2$ such that the
group~$\Sym(X)$ is realized by a group of isometries of the
Fubini-Study metric.

Let us divide the complex projective plane~$\CP^2$ into three
four-dimensional disks
\begin{gather*}
B_1=\{(z_1:z_2:z_3)\mid |z_1|\ge |z_2|\text{ and }|z_1|\ge
|z_3|\};\\
B_2=\{(z_1:z_2:z_3)\mid |z_2|\ge |z_1|\text{ and }|z_2|\ge
|z_3|\};\\
B_3=\{(z_1:z_2:z_3)\mid |z_3|\ge |z_1|\text{ and }|z_3|\ge
|z_2|\}.
\end{gather*}
The intersections $\Pi_1=B_2\cap B_3$, $\Pi_2=B_3\cap B_1$, and
$\Pi_3=B_1\cap B_2$ are solid tori. The intersection $T=B_1\cap
B_2\cap B_3$ is the two-dimensional torus given by the equation
$|z_1|=|z_2|=|z_3|$. The boundary of each disk~$B_j$ is divided by
the torus~$T$ into two solid tori~$\Pi_k$ and~$\Pi_l$, where
$\{j,k,l\}=\{1,2,3\}$. In the paper~\cite{BaKu92} T.\,F.\,Banchoff
and W.\,K\"uhnel introduced the notion of an \textit{equilibrium
triangulation} of~$\CP^2$. According to their definition, a
triangulation of the complex projective plane is said to be
equilibrium if the disks~$B_j$, the solid tori~$\Pi_j$, and the
torus~$T$ are the subcomplexes of this triangulation.
T.\,F.\,Banchoff and W.\,K\"uhnel constructed a series of
equilibrium triangulations of~$\CP^2$ with interesting
automorphism groups. The simplest of these triangulations has 10
vertex and the automorphism group of order~42. The
triangulation~$X$ is equilibrium too, though it does not belong to
the series constructed by T.\,F.\,Banchoff and W.\,K\"uhnel.

For each two-dimensional simplex of~$X$ the number of
four-dimensional simplices containing it is either~$4$ or~$6$. It
is interesting to compare this fact with a result of N.\,Brady,
J.\,McCammond, and J.\,Meier on triangulations of
three-dimensional manifolds. In~\cite{BMM04} they proved that
every closed three-dimensional manifold has a triangulation in
which each edge is contained in exactly $4$, $5$, or $6$
tetrahedra. Besides it is known that every three-dimensional
manifold admitting a triangulation in which each edge is contained
in not more than $5$ tetrahedra can be covered by the
three-dimensional sphere (see~\cite{Mat06},~\cite{Lut06}). Hence
it is interesting to ask which four-dimensional manifolds possess
triangulations such that each two-dimensional simplex is contained
in exactly $4$, $5$, or $6$ four-dimensional simplices. It is also
interesting for which four-dimensional manifolds one can dispense
with one of the three numbers $4$, $5$, and $6$.

Another interesting problem closely related with the problem of
constructing triangulations of manifolds is the problem of
constructing triangulations of mappings of manifolds. Under a
\textit{triangulation} of a mapping~$g:M\to N$ we mean a pair of
triangulations of the manifolds~$M$ and~$N$ with respect to which
the mapping~$g$ is simplicial. (Recall that a mapping of
simplicial complexes is called \textit{simplicial} if it takes
vertices of the first complex to vertices of the second complex
and maps linearly every simplex of the first complex onto some
simplex of the second complex.) For example, K.\,V.\,Madahar and
K.\,S.\,Sarkaria~\cite{MaSa00} constructed a minimal by the number
of vertices triangulation of the Hopf mapping~$S^3\to S^2$. We
shall consider the classical moment mapping~$\mu:\CP^2\to\Delta^2$
given by
\begin{equation}
\label{eq_moment}
\mu(z_1:z_2:z_3)=\frac{\bigl(|z_1|^2,|z_2|^2,|z_3|^2\bigr)}{|z_1|^2+|z_2|^2+|z_3|^2},
\end{equation}
where $\Delta^2\subset\R^3$ is the equilateral triangle with
vertices~$(1,0,0)$, $(0,1,0)$, and~$(0,0,1)$. In
section~\ref{section_moment} we shall use the triangulation~$X$ to
construct a triangulation of the moment mapping~$\mu$. For a
triangulation of the complex projective plane we shall take a
33-vertex $\Sym(X)$-invariant subdivision~$\bX$ of~$X$, for a
triangulation of the triangle~$\Delta^2$ we shall take its
barycentric subdivision.

Like K\"uhnel's triangulation~$K$, the triangulation~$X$ is
related with complex crystallographic groups. This relationship is
obtained in section~\ref{section_crystal}.

\section{Simplicial complex~$X$}
\label{section_constr}

An \textit{abstract simplicial complex} on a vertex set~$V$ is a
set~$K$ of finite subsets of~$V$ such that $\varnothing\in K$ and
for any subsets~$\tau\subset\sigma\subset V$ if $\sigma$ belongs
to~$K$ and~$\tau\subset\sigma$, then~$\tau$ belongs to~$K$ too.
The geometric realization of a simplicial complex~$K$ is denoted
by~$|K|$. A \textit{full subcomplex} of~$K$ spanned by a vertex
set~$W\subset V$ is a simplicial complex consisting of all
simplices~$\sigma\in K$ such that $\sigma\subset W$. The
\textit{link} of a simplex~$\sigma\in K$ is the subcomplex of~$K$
consisting of all simplices~$\tau$ such
that~$\tau\cap\sigma=\varnothing$ and $\tau\cup\sigma\in K$. A
simplicial complex is said to be an $n$-dimensional
\textit{combinatorial sphere} if its geometric realization is
piecewise linearly homeomorphic to the boundary of an
$(n+1)$-dimensional simplex. A simplicial complex is said to be an
$n$-dimensional \textit{combinatorial manifold} if the links of
all its vertices are $(n-1)$-dimensional combinatorial spheres. A
\textit{simplicial mapping} of a simplicial complex~$K_1$ on
vertex set~$V_1$ to a simplicial complex~$K_2$ on vertex set~$V_2$
is a mapping $f:V_1\to V_2$ such that $f(\sigma)\in K_2$ whenever
$\sigma\in K_1$. An \textit{isomorphism} of simplicial complexes
is a simplicial mapping with a simplicial inverse.

\begin{constr}
We construct a $15$-vertex abstract simplicial complex~$X$ in the
following way. For the vertex set of~$X$ we take the $15$-element
set
$$
V=\bigl(V_4\setminus\{e\}\bigr)\sqcup\bigl(\{1,2,3,4\}\times\{1,2,3\}\bigr),
$$
where $V_4\subset S_4$ is the Klein four group. Thus the vertices
of~$X$ are the permutations~$(12)(34)$, $(13)(24)$, and $(14)(23)$
and the pairs of integers~$(a,b)$ such that $1\le a\le 4$ and
$1\le b\le 3$. The four-dimensional simplices of~$X$ are spanned
by the sets
$$
\nu,(1,b_1),(2,b_2),(3,b_3),(4,b_4),\qquad\nu\in
V_4\setminus\{e\}, 1\le b_a\le 3, a=1,2,3,4,
$$
such that $b_{\nu(a)}\ne b_a$ for $a=1,2,3,4$. The simplices of
dimensions less than $4$ are spanned by all subsets of the above
sets.
\end{constr}

The vertices of~$X$ admit a regular colouring in 5 colours. To
obtain such colouring one should paint the vertices $\nu\in
V_4\setminus\{e\}$ in colour~$0$ and the vertices $(a,b)$ in
colour~$a$, $a=1,2,3,4$. It is easy to compute that the $f$-vector
of the triangulation~$X$ is equal to $(15,90,240,270,108)$. Later
we shall prove that~$X$ is a combinatorial manifold. In
sections~\ref{section_explicit} and~\ref{section_crystal} we shall
prove in two different ways that the geometric realization~$|X|$
is piecewise linearly homeomorphic to~$\CP^2$.

Now we compute the automorphism group of~$X$. We define actions of
the groups~$S_4$ and~$S_3$ on the set~$V$ by
\begin{align*}
\theta\cdot\nu &=\theta\nu\theta^{-1},& \theta\cdot (a,b)
&=(\theta(a),b), &\theta\in S_4;\\
\varkappa\cdot\nu &=\nu,& \varkappa\cdot (a,b) &=(a,\varkappa(b)),
&\varkappa\in S_3.
\end{align*}
Obviously, these actions take simplices of~$X$ to simplices
of~$X$. Besides, these actions commute. Hence the group~$S_4\times
S_3$ acts on~$X$ by automorphisms. Obviously, the kernel of this
action is trivial. Therefore the automorphism group~$\Sym(X)$
of~$X$ contains the subgroup isomorphic to~$S_4\times S_3$.
Indeed, this subgroup coincides with the whole group~$\Sym(X)$.

\begin{propos}
The automorphism group $\Sym(X)$ of the simplicial complex~$X$ is
isomorphic to~$S_4\times S_3$.
\end{propos}

The vertices of~$X$ are divided into two $\Sym(X)$-orbits. The
first orbit consists of the $3$ vertices $(12)(34)$, $(13)(24)$,
and $(14)(23)$. The second orbit consists of $12$
vertices~$(a,b)$.

\begin{propos}
The simplicial complex~$X$ is a four-dimensional combinatorial
manifold.
\end{propos}
\begin{proof}
We need to prove that the links of all vertices of~$X$ are
three-dimensional combinatorial spheres. Obviously, the links of
all vertices in the same $\Sym(X)$-orbit are pairwise isomorphic.
Thus we suffice to prove that the links of the vertices~$(12)(34)$
and~$(1,1)$ are three-dimensional combinatorial spheres.

The link of the vertex~$(12)(34)$ consists of the simplices
$$
(1,b_1),(2,b_2),(3,b_3),(4,b_4)
$$
such that $b_1\ne b_2$ и $b_3\ne b_4$. Hence it is isomorphic to
the join of two closed $6$-vertex circles shown in
Fig.~\ref{fig_link(12)(34)}.
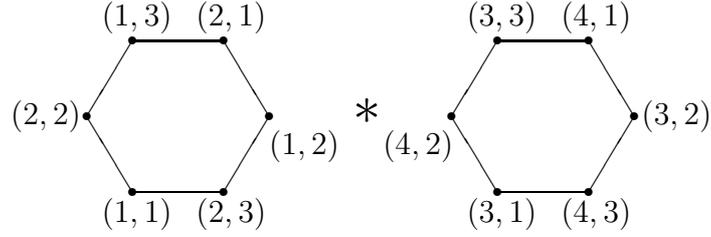
\begin{figure}
\unitlength=0.2cm
\begin{picture}(48,14)
\multiput(12,7)(24,0){2}{\begin{picture}(0,0)

\put(-6,0){\circle*{0.5}}

\put(6,0){\circle*{0.5}}

\put(-3,5){\circle*{0.5}}

\put(-3,-5){\circle*{0.5}}

\put(3,5){\circle*{0.5}}

\put(3,-5){\circle*{0.5}}

\put(-3,5){\line(1,0){6}}

\put(-3,-5){\line(1,0){6}}

\put(-3,5){\line(-3,-5){3}}

\put(-3,-5){\line(-3,5){3}}

\put(3,5){\line(3,-5){3}}

\put(3,-5){\line(3,5){3}}

\end{picture}}

{\LARGE\put(23.5,6.5){$*$}}

\put(7,0){$(1,1)$}

\put(13.2,0){$(2,3)$}

\put(7,13){$(1,3)$}

\put(13.2,13){$(2,1)$}

\put(1,6.5){$(2,2)$}

\put(18,4.5){$(1,2)$}

\put(31,0){$(3,1)$}

\put(37.2,0){$(4,3)$}

\put(31,13){$(3,3)$}

\put(37.2,13){$(4,1)$}

\put(25.5,4.5){$(4,2)$}

\put(42.5,6.5){$(3,2)$}

\end{picture}
\caption{The link of the vertex $(12)(34)$}
\label{fig_link(12)(34)}
\end{figure}
\begin{figure}
\unitlength=5mm
\begin{picture}(22,16)
\put(4,8){\circle*{0.2}}

\put(7,8){\circle*{0.2}}

\put(15,8){\circle*{0.2}}

\put(18,8){\circle*{0.2}}

\put(9,4){\circle*{0.2}}

\put(9,12){\circle*{0.2}}

\put(13,4){\circle*{0.2}}

\put(13,12){\circle*{0.2}}

\put(0,8){\line(1,0){7}}

\put(15,8){\line(1,0){7}}

\put(4,8){\line(5,4){5}}

\put(4,8){\line(5,-4){5}}

\put(7,8){\line(3,-2){6}}

\put(7,16){\line(1,-2){8}}

\put(7,0){\line(1,2){2}}

\put(13,12){\line(1,2){2}}

\put(9,12){\line(3,-2){6}}

\put(13,4){\line(5,4){5}}

\put(13,12){\line(5,-4){5}}

\multiput(7.15,8.1)(1.2,0.8){5}{\line(3,2){0.8}}

\multiput(9.15,4.1)(1.2,0.8){5}{\line(3,2){0.8}}

\multiput(9.1,4.2)(0.6,1.2){8}{\line(1,2){0.4}}

\footnotesize

\put(2.5,7.1){$(3,1)$}

\put(4.9,7.3){\scriptsize$(14)(23)$}

\put(18,7.1){$(4,1)$}

\put(14.8,7.3){\scriptsize$(13)(24)$}

\put(7,12){$(4,2)$}

\put(13.5,12){$(3,3)$}

\put(7,3.7){$(4,3)$}

\put(13.5,3.7){$(3,2)$}

\thicklines

\put(7,8){\line(1,2){2}}

\put(7,8){\line(1,-2){2}}

\put(15,8){\line(-1,2){2}}

\put(15,8){\line(-1,-2){2}}

\put(9,4){\line(1,0){4}}

\put(9,12){\line(1,0){4}}

\end{picture}
\caption{The subcomplex $J$ of the link of the vertex $(1,1)$}
\label{fig_J}
\end{figure}
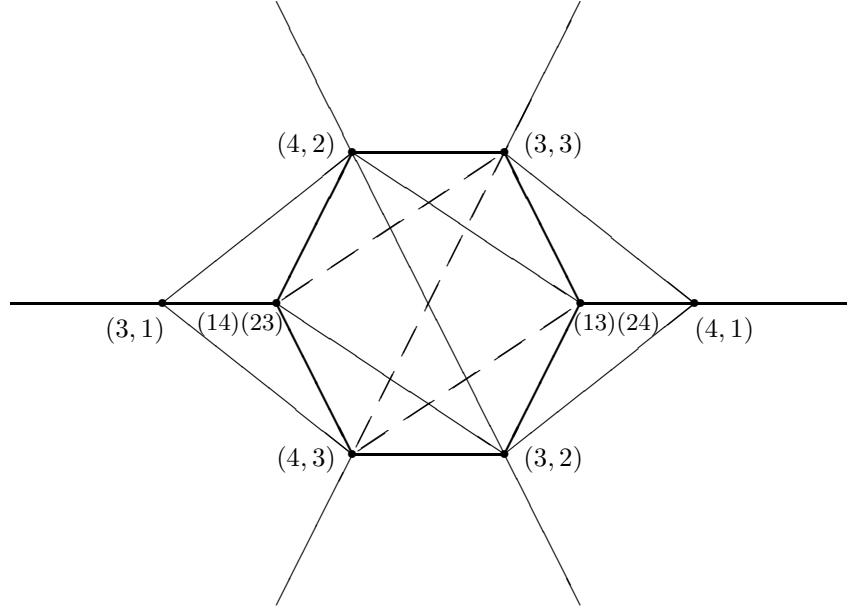

The link of the vertex~$(1,1)$ is more complicated. We denote it
by~$\N$. First, let us consider the full subcomplex~$J$ of~$\N$
spanned by the set of all vertices of~$\N$ except for the
vertices~$(2,1)$,~$(2,2)$, and~$(2,3)$. This complex can be
realized in a three-dimensional sphere as it is shown in
Fig~\ref{fig_J}. (Here the three-dimensional sphere is identified
with the one-point compactification of~$\R^3$.) We imply that the
vertex~$(12)(34)$ is placed at the infinitely remote point, the
vertices~~$(3,2)$ and~$(4,2)$ are somewhat ``lifted'' above the
plane of the figure and the vertices~$(3,3)$ and~$(4,3)$ are
somewhat ``lowered'' below the plane of the figure. The
complex~$J$ is the union of three two-dimensional disks. The first
disk consists of the triangles
\begin{align*}
&(12)(34),(3,1),(4,2);&&(12)(34),(3,3),(4,2);\\
&(12)(34),(3,1),(4,3);&&(13)(24),(3,2),(4,1);\\
&(12)(34),(3,2),(4,3);&&(13)(24),(3,3),(4,1);\\
&(12)(34),(3,2),(4,1);&&(14)(23),(3,1),(4,2);\\
&(12)(34),(3,3),(4,1);&&(14)(23),(3,1),(4,3).
\end{align*}
The second disk consists of the triangles
\begin{align*}
&(13)(24),(3,3),(4,2);&&(14)(23),(3,2),(4,2);\\
&(13)(24),(3,2),(4,2);&&(14)(23),(3,2),(4,3).
\end{align*}
The third disk consists of the triangles
\begin{align*}
&(13)(24),(3,2),(4,3);&&(14)(23),(3,3),(4,3);\\
&(13)(24),(3,3),(4,3);&&(14)(23),(3,3),(4,2).
\end{align*}
The complement of~$J$ in the three-dimensional sphere is the union
of three open three-dimensional disks. The first disk~$D_1$ is the
interior of the octahedron with vertices~$(13)(24)$, $(14)(23)$,
$(3,2)$, $(3,3)$, $(4,2)$,~$(4,3)$, the second disk~$D_2$ is
situated ``below the plane of the figure'', and the third
disk~$D_3$ is situated ``above the plane of the figure''. We place
the vertices~$(2,1)$, $(2,2)$, and~$(2,3)$ in the interiors of the
disks~$D_1$,~$D_2$, and~$D_3$ respectively and triangulate the
disks~$D_1$,~$D_2$, and~$D_3$ as cones over their boundaries with
vertices~$(2,1)$, $(2,2)$, and~$(2,3)$ respectively. The obtained
triangulation of the three-dimensional sphere is isomorphic to the
simplicial complex~$\N$.
\end{proof}

Let us now describe the links of edges and two-dimensional
simplices of~$X$. The edges of~$X$ are divided into three
$\Sym(X)$-orbits whose representatives are the edge with vertices
$(1,1)$ and $(2,1)$, the edge with vertices $(12)(34)$ and
$(1,1)$, and the edge with vertices $(1,1)$ and $(2,2)$
respectively. The first orbit consists of $18$ edges, either of
the second and the third orbits consists of~$36$ edges. The link
of an edge in the first orbit is isomorphic to the boundary of an
octahedron. The link of an edge in the second orbit is isomorphic
to the suspension over the boundary of a hexagon. The link of an
edge in the third orbit is isomorphic to the triangulation of a
two-dimensional sphere that can be obtained by gluing together
two examples of the triangulation shown in Fig.~\ref{fig_linkedge}
along their boundaries.

\begin{figure}
\unitlength=4mm
\begin{picture}(12,10)
\put(0,0){\circle*{0.2}}

\put(12,0){\circle*{0.2}}

\put(6,10){\circle*{0.2}}

\put(4,4){\circle*{0.2}}

\put(8,4){\circle*{0.2}}

\put(6,1.5){\circle*{0.2}}

\put(0,0){\line(1,0){12}}

\put(0,0){\line(3,5){6}}

\put(0,0){\line(4,1){6}}

\put(0,0){\line(1,1){4}}

\put(4,4){\line(4,-5){2}}

\put(8,4){\line(-4,-5){2}}

\put(4,4){\line(1,0){4}}

\put(4,4){\line(1,3){2}}

\put(8,4){\line(-1,3){2}}

\put(6,10){\line(3,-5){6}}

\put(12,0){\line(-1,1){4}}

\put(12,0){\line(-4,1){6}}

\put(0,0){\line(1,0){12}}

\put(0,0){\line(1,0){12}}

\put(0,0){\line(1,0){12}}

\put(0,0){\line(1,0){12}}

\end{picture}

\caption{Half the link of an edge of~$X$ in the third
$\Sym(X)$-orbit} \label{fig_linkedge}
\end{figure}
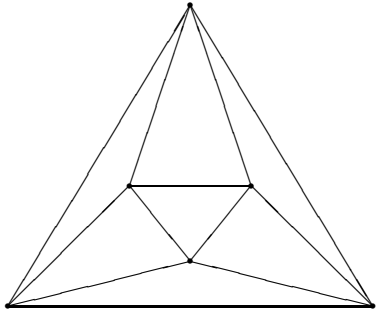

The link of a codimension~$2$ simplex of a combinatorial manifold
is always isomorphic to the boundary of a polygon and is
completely characterized by the number of vertices of this
polygon, which is equal to the number of maximal-dimensional
simplices containing this codimension~$2$ simplex. The
two-dimensional simplices of~$X$ are divided into five
$\Sym(X)$-orbits consisting of $36$, $36$, $72$, $72$, and $24$
simplices respectively with representatives
$\{(12)(34),(1,1),(2,2)\}$, $\{(12)(34),(1,1),(3,1)\}$,
$\{(12)(34),(1,1),(3,2)\}$, $\{(1,1),(2,1),(3,2)\}$, and
$\{(1,1),(2,2),(3,3)\}$ respectively. The links of simplices in
the first orbit and in the fifth orbit are isomorphic to the
boundary of a hexagon, the links of simplices in the second orbit,
the third orbit, and in the fourth orbit are isomorphic to the
boundary of a quadrangle. Thus every two-dimensional simplex
of~$X$ is contained either in $4$ or in $6$ four-dimensional
simplices.

Let us now describe a $33$-vertex $\Sym(X)$-invariant
subdivision~$\bX$ of~$X$, which will be needed in
section~\ref{section_moment} to construct a triangulation of the
moment mapping. As it has been mentioned before the set of edges
of~$X$ is divided into three $\Sym(X)$-orbits. Consider the orbit
consisting of $18$ edges~$(a_1,b)$,~$(a_2,b)$, $a_1\ne a_2$. At
the midpoint of every such edge we introduce a new vertex denoted
by~$(\widehat{a_1a_2},b)$. (Notations~$(\widehat{a_1a_2},b)$
and~$(\widehat{a_2a_1},b)$ correspond to the same vertex.) The set
of four-dimensional simplices of~$X$ is divided into two
$\Sym(X)$-orbits. The first orbit consists of $72$ simplices
\begin{equation}
\label{eq_s1}
(a_1a_2)(a_3a_4),(a_1,b_1),(a_2,b_2),(a_3,b_1),(a_4,b_3),
\end{equation}
where the numbers $a_1$, $a_2$, $a_3$, and $a_4$ are pairwise
distinct and the numbers $b_1$, $b_2$, and $b_3$ are pairwise
distinct. The second orbit consists of $36$ simplices
\begin{equation}
\label{eq_s2}
(a_1a_2)(a_3a_4),(a_1,b_1),(a_2,b_2),(a_3,b_1),(a_4,b_2),
\end{equation}
where the numbers $a_1$, $a_2$, $a_3$, and $a_4$ are pairwise
distinct and $b_1\ne b_2$. We divide every four-dimensional
simplex~(\ref{eq_s1}) into $2$ four-dimensional simplices
\begin{equation}
\label{eq_bX1}
\begin{aligned}
(a_1a_2)(a_3a_4),(\widehat{a_1a_3},b_1),(a_1,b_1),(a_2,b_2),(a_4,b_3),\\
(a_1a_2)(a_3a_4),(\widehat{a_1a_3},b_1),(a_3,b_1),(a_2,b_2),(a_4,b_3)
\end{aligned}
\end{equation}
and divide every four-dimensional simplex~(\ref{eq_s2}) into $4$
four-dimensional simplices
\begin{equation}
\label{eq_bX2}
\begin{aligned}
(a_1a_2)(a_3a_4),(\widehat{a_1a_3},b_1),(\widehat{a_2a_4},b_2)(a_1,b_1),(a_2,b_2),\\
(a_1a_2)(a_3a_4),(\widehat{a_1a_3},b_1),(\widehat{a_2a_4},b_2)(a_1,b_1),(a_4,b_2),\\
(a_1a_2)(a_3a_4),(\widehat{a_1a_3},b_1),(\widehat{a_2a_4},b_2)(a_3,b_1),(a_2,b_2),\\
(a_1a_2)(a_3a_4),(\widehat{a_1a_3},b_1),(\widehat{a_2a_4},b_2)(a_3,b_1),(a_4,b_2).
\end{aligned}
\end{equation}
As a result we obtain a $33$-vertex $\Sym(X)$-invariant
subdivision~$\bX$ of~$X$ with $288$ four-dimensional simplices. It
can be easily checked that the four-dimensional simplices of~$X$
are divided into two $\Sym(X)$-orbits. The first orbit consists
of~$144$ simplices~(\ref{eq_bX1}) and the second orbit consists
of~$144$ simplices~(\ref{eq_bX2}).

\section{Minimal triangulations and chess colourings of simplices}
\label{section_colour} Let $M$ be a closed $n$-dimensional
manifold.

\begin{defin}
A \textit{triangulation} of a manifold $M$ is a simplicial
complex~$K$ together with a homeomorphism~$\varphi:|K|\to M$. A
triangulation is said to be \textit{combinatorial} or
\textit{piecewise linear} if~$K$ is a combinatorial manifold.
\end{defin}

In this paper we shall deal only with manifolds of dimension not
greater than~$4$. It is well known that in this case every
triangulation is combinatorial. We shall often not care about a
concrete homeomorphism~$\varphi$ and say that a simplicial
complex~$K$ is a triangulation of~$M$ if $|K|\approx M$.

By~$\T(M)$ we denote the set of all triangulations of~$M$ up to an
isomorphism. The set~$\T(M)$ contains the following interesting
subsets.

1. $\T_{even}(M)\subset\T(M)$ is the subset consisting of all
triangulations~$K$ such that every $(n-2)$-dimensional simplex
of~$K$ is contained in even number of $n$-dimensional simplices.

2. $\T_{bw}(M)\subset \T(M)$ is the subset consisting of all
triangulations admitting a chess colouring of $n$-dimensional
simplices, that is, a colouring in black and white colours such
that any two simplices possessing a common facet have distinct
colours.

3. $\T_{colour}(M)\subset \T(M)$ is the subset consisting of all
triangulations~$K$ admitting a regular colouring of vertices, that
is, a colouring of vertices in colours of some $(n+1)$-element set
such that every $n$-dimensional simplex of~$K$ contains exactly
one vertex of each colour. (Usually we denote the colours by the
numbers $0,1,\ldots,n$.)

The classes of triangulations~$\T_{even}(M)$, $\T_{bw}(M)$,
and~$\T_{colour}(M)$ are very important in theory of
discretization of differential geometric connections and complex
analysis developed in works of I.\,A.\,Dynnikov and
S.\,P.\,Novikov~\cite{NoDy97}--\cite{DyNo03}. In these papers an
\textit{operator of discrete connection} on a triangulation~$K$ of
a manifold~$M$ is defined to be an operator that takes a function
on the vertex set of~$K$ to a function on the set of
$n$-dimensional simplices of~$K$ by
\begin{equation}
\label{eq_DC}
(Q\psi)_{\sigma}=\sum_{v\in\sigma}b_{\sigma,v}\psi_v,
\end{equation}
where~$b_{\sigma,v}$ is a fixed set of coefficients. A discrete
connection is called \textit{canonical} if all
coefficients~$b_{\sigma,v}$ are equal to~$1$. We shall always deal
with canonical discrete connections only. A \textit{fat path} in a
triangulation~$K$ is a sequence of $n$-dimensional simplices such
that any two consecutive simplices have a common facet. Solving
equation~(\ref{eq_DC}) along the fat path corresponding to the
circuite around an $(n-2)$-dimensional simplex~$\tau$, we obtain
the \textit{local holonomy} or the \textit{curvature} of a
discrete connection at the simplex~$\tau$. For the canonical
discrete connection the local holonomy at simplex~$\tau$ is
trivial if and only if the simplex~$\tau$ is contained in even
number of $n$-dimensional simplices. Hence $\T_{even}(M)$ is
exactly the class of all triangulations whose canonical discrete
connections have zero curvature. If the local holonomy is trivial,
then the \textit{global holonomy} homomorphism $\rho:\pi_1(M)\to
S_{n+1}$ is well defined. It is easy to see that
$\T_{colour}(M)\subset\T_{even}(M)$ is exactly the subclass
consisting of all triangulations whose canonical discrete
connections have trivial global holonomy.

Following~\cite{DyNo03} we denote by $\rho_1$ the composite
homomorphism
$$
\pi_1\xrightarrow{\rho}S_{n+1}\xrightarrow{\sgn}\Z_2.
$$
Two other homomorphisms considered in~\cite{DyNo03} are the
orientation homomorphism~$\rho_2:\pi_1(M)\to\Z_2$ and the
homomorphism~$\rho_3:\pi_1(M)\to\Z_2$ that takes each homotopy
class to the parity of the number of $n$-dimensional simplices in
a fat path representing this homotopy class. (The
homomorphism~$\rho_3$ is well defined if the triangulation belongs
to~$\T_{even}(M)$.) Obviously, the
subclass~$\T_{bw}(M)\subset\T_{even}(M)$ consists of all
triangulations for which the homomorphism~$\rho_3$ is trivial. As
it was mentioned in~\cite{DyNo03}, the homomorphisms
~$\rho_1,\rho_2,\rho_3$ satisfy the
relation~$\rho_1\rho_2=\rho_3$, where we use the multiplicative
notation for the group~$\Z_2$. It easily follows from this
relation that $\T_{colour}(M)\subset\T_{bw}(M)$ if the
manifold~$M$ is orientable and
$\T_{colour}(M)\cap\T_{bw}(M)=\varnothing$ if the manifold~$M$ is
non-orientable.

Notice that the holonomy group of the canonical discrete
connection coincides with the \textit{projectivity group} of the
triangulation, which was introduced by
M.\,Joswig~\cite{Jos01},~\cite{Jos02}. Hence in M.\,Joswig's
terminology the class~$\T_{colour}(M)$ is exactly the class of all
triangulations with trivial projectivity groups.

It follows from the above reasoning that
$\T_{colour}(M)=\T_{bw}(M)=\T_{even}(M)$ if $M$ is
simply-connected. Moreover, $\T_{bw}(M)=\T_{even}(M)$ if
$\pi_1(M)$ has no non-trivial homomorphisms to~$\Z_2$ and
$\T_{colour}(M)=\T_{even}(M)$ if $\pi_1(M)$ has no non-trivial
homomorphisms to~$S_{n+1}$.

Triangulations  with regular colourings of vertices appear as well
in many other problems. For example, in the paper~\cite{DaJa91} by
M.\,W.\,Davis and T.\,Januszkiewicz such triangulations lead to an
important class of manifolds called \textit{small coverings
induced from a linear model over simple polytopes} and in the
author's paper~\cite{Gai08} such triangulations are important for
a construction of combinatorial realization of cycles.

Two-dimensional triangulations with chess colourings of triangles
are very important for the discretization of complex
analysis~\cite{DyNo03}.

A \textit{minimal} triangulation of a manifold~$M$ is a
triangulation of~$M$ with the smallest possible number of
vertices. The problem of finding minimal triangulations of
manifolds is a very interesting and hard problem. It is solved
only for a few manifolds. A good survey of results on minimal
triangulations is the paper~~\cite{Lut05} by F.\,Lutz. For the
complex projective plane a minimal triangulation has 9 vertex and
is unique up to an isomorphism. It was constructed by W.\,K\"uhnel
in 1980 (see~\cite{KuBa83},~\cite{KuLa83}). In the present paper
we are interested in the problems of finding minimal
triangulations of manifolds in the classes~$\T_{even}$,~$\T_{bw}$,
and~$\T_{colour}$. For simply-connected manifolds these three
problems coincide.

\begin{ex}
Let~$e_0,e_1,\ldots,e_n$ be the standard basis in~$\R^{n+1}$. The
convex hull of the points~$\pm e_0,\pm e_1,\ldots,\pm e_n$ is
called a \textit{cross-polytope}. It is the regular polytope dual
to the $(n+1)$-dimensional cube. The boundary of the
cross-polytope is a triangulation of the $n$-dimensional sphere.
This triangulation admits a regular colouring of vertices. To
obtain such colouring one should paint the vertices~$\pm e_j$ in
colour~$j$. Obviously, this triangulation is minimal in the
class~$\T_{even}(S^n)=\T_{bw}(S^n)=\T_{colour}(S^n)$. Indeed, if a
triangulation of an $n$-dimensional manifold admits a regular
colouring of vertices, then this triangulation should contain at
least $2$ vertices of each colour and, hence, at least~$2n+2$
vertices.
\end{ex}

We have the following simple proposition.

\begin{propos}
\label{propos_3n+3} If a combinatorial triangulation of an
$n$-dimensional manifold admits a regular colouring of vertices
in~$n+1$ colours and contains less than~$3n+3$ vertices, then this
manifold is piecewise linearly homeomorphic to an $n$-dimensional
sphere.
\end{propos}
\begin{proof}
If there is a colour with only two vertices $v_1$ and $v_2$
coloured by it, then the triangulation is the suspension with
vertices $v_1$ and $v_2$ over the full subcomplex spanned by all
other vertices of the triangulation. A combinatorial manifold that
is a suspension is piecewise linearly homeomorphic to a sphere.
\end{proof}

\begin{ex}
It is well known that a minimal triangulation of the
two-dimensional torus is unique up to an isomorphism and has~$7$
vertices. The most symmetric realization of this triangulation is
shown in Fig.~\ref{fig_T2},\textit{a}. (We imply the
identification of the opposite sides of the hexagon.) One can
easily check that this triangulation belongs to the
classes~$\T_{even}(T^2)$ and~$\T_{bw}(T^2)$ and does not belong to
the class~$\T_{colour}(T^2)$. Proposition~\ref{propos_3n+3}
implies that a minimal triangulation in the
class~$\T_{colour}(T^2)$ has at least~$9$ vertices. Actually such
triangulation has $9$ vertices and is unique up to an isomorphism.
This triangulation is shown in Fig.~\ref{fig_T2},\textit{b}.
\end{ex}
\begin{ex}
A minimal triangulation of the real projective plane~$\RP^2$ is
unique, has~$6$ vertices, and can be obtained from the boundary of
a regular icosahedron by identifying every pair of antipodal
points. This triangulation belongs to none of the classes
~$\T_{even}(\RP^2)$,~$\T_{bw}(\RP^2)$, and~$\T_{colour}(\RP^2)$. A
minimal triangulation in the classes~$\T_{even}(\RP^2)$
and~$\T_{bw}(\RP^2)$ is unique and has $7$ vertices (see
Fig.~\ref{fig_RP2},\textit{a}). A minimal triangulation in the
class~~$\T_{colour}(\RP^2)$ is also unique and has $9$ vertices
(see Fig.~\ref{fig_RP2},\textit{b}).
\end{ex}

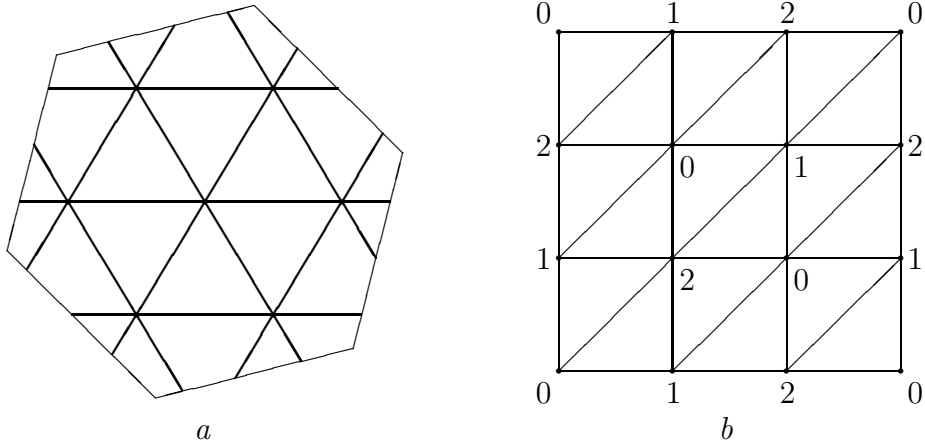
\begin{figure}
\unitlength=3mm
\begin{picture}(38,19.2)

\put(7,0){\textit{a}}

\put(30,0){\textit{b}}

\put(7.5,10.5){\begin{picture}(0,0)

\put(0,0){\circle*{0.2}}

\put(-6,0){\circle*{0.2}}

\put(-3,5){\circle*{0.2}}

\put(-3,-5){\circle*{0.2}}

\put(3,-5){\circle*{0.2}}

\put(3,5){\circle*{0.2}}

\put(6,0){\circle*{0.2}}

\put(-6.5,6.5){\line(-1,-4){2.167}}

\put(6.5,-6.5){\line(1,4){2.167}}

\put(-8.667,-2.167){\line(1,-1){6.5}}

\put(8.667,2.167){\line(-1,1){6.5}}

\put(-6.5,6.5){\line(4,1){8.667}}

\put(6.5,-6.5){\line(-4,-1){8.667}}

\thicklines

\put(-8.13,0){\line(1,0){16.26}}

\put(0,0){\line(3,5){4.05}}

\put(0,0){\line(-3,-5){4.05}}

\put(0,0){\line(3,-5){4.23}}

\put(0,0){\line(-3,5){4.23}}

\put(-6.88,5){\line(1,0){12.72}}

\put(6.88,-5){\line(-1,0){12.72}}

\put(-6,0){\line(3,5){4.67}}

\put(6,0){\line(-3,-5){4.67}}

\put(-6,0){\line(3,-5){5.01}}

\put(6,0){\line(-3,5){5.01}}

\put(-6,0){\line(-3,-5){1.8}}

\put(-6,0){\line(-3,5){1.5}}

\put(6,0){\line(3,5){1.8}}

\put(6,0){\line(3,-5){1.5}}

\end{picture}}

\put(23,3){\begin{picture}(0,0)

\multiput(0,0)(5,0){4}{\multiput(0,0)(0,5){4}{\circle*{.2}}}

\multiput(0,0)(5,0){4}{\line(0,1){15}}

\multiput(0,0)(0,5){4}{\line(1,0){15}}

\put(0,0){\line(1,1){15}}

\put(0,5){\line(1,1){10}}

\put(5,0){\line(1,1){10}}

\put(10,0){\line(1,1){5}}

\put(0,10){\line(1,1){5}}

\put(-1,-1.4){0}

\put(-1,15.4){0}

\put(15.3,-1.4){0}

\put(15.3,15.4){0}

\put(-1,4.5){1}

\put(-1,9.5){2}

\put(15.3,4.5){1}

\put(15.3,9.5){2}

\put(4.7,-1.4){1}

\put(9.7,-1.4){2}

\put(4.7,15.4){1}

\put(9.7,15.4){2}

\put(5.3,3.6){2}

\put(10.3,3.6){0}

\put(5.3,8.6){0}

\put(10.3,8.6){1}

\end{picture}}

\end{picture}
\caption{Triangulations of the two-dimensional torus:
(\textit{a})~$7$-vertex; (\textit{b}) $9$-vertex} \label{fig_T2}
\end{figure}

\begin{figure}
\unitlength=5mm
\begin{picture}(25,12.8)
\put(5,7){\begin{picture}(0,0)

\put(-5,0){\circle*{0.2}}

\put(-2.5,0){\circle*{0.2}}

\put(2.5,0){\circle*{0.2}}

\put(5,0){\circle*{0.2}}

\put(0,2.5){\circle*{0.2}}

\put(0,-2.5){\circle*{0.2}}

\put(-2.5,5){\circle*{0.2}}

\put(-2.5,-5){\circle*{0.2}}

\put(2.5,5){\circle*{0.2}}

\put(2.5,-5){\circle*{0.2}}

\put(-5,0){\line(1,0){10}}

\put(-2.5,-5){\line(0,1){10}}

\put(2.5,-5){\line(0,1){10}}

\put(-2.5,0){\line(1,1){5}}

\put(-2.5,0){\line(1,-1){5}}

\put(2.5,0){\line(-1,1){5}}

\put(2.5,0){\line(-1,-1){5}}

\thicklines

\put(-5,0){\line(1,2){2.5}}

\put(-5,0){\line(1,-2){2.5}}

\put(5,0){\line(-1,2){2.5}}

\put(5,0){\line(-1,-2){2.5}}

\put(-2.5,5){\line(1,0){5}}

\put(-2.5,-5){\line(1,0){5}}

\put(5,0){\vector(-1,2){1.3}}

\put(-5,0){\vector(1,-2){1.3}}

\put(2.5,5){\vector(-1,0){2.45}}

\put(2.5,5){\vector(-1,0){2.85}}

\put(-2.5,-5){\vector(1,0){2.45}}

\put(-2.5,-5){\vector(1,0){2.85}}

\put(-2.5,5){\vector(-1,-2){1.1}}

\put(-2.5,5){\vector(-1,-2){1.3}}

\put(-2.5,5){\vector(-1,-2){1.5}}

\put(2.5,-5){\vector(1,2){1.1}}

\put(2.5,-5){\vector(1,2){1.3}}

\put(2.5,-5){\vector(1,2){1.5}}

\it{\Large

\put(-3.7,-1.7){b}

\put(2.9,-1.7){b}

\put(-3.7,1){w}

\put(2.9,1){w}

\put(-1.8,-2.9){w}

\put(-1.8,2.1){b}

\put(1.1,-2.9){w}

\put(1.1,2.1){b}

\put(-.3,.7){w}

\put(-.3,3.5){w}

\put(-.3,-4.3){b}

\put(-.3,-1.4){b}}

\put(-0.25,-7){a}

\end{picture}}

\put(20,7){\begin{picture}(0,0)

\put(-5,0){\circle*{0.2}}

\put(0,0){\circle*{0.2}}

\put(5,0){\circle*{0.2}}

\put(0,5){\circle*{0.2}}

\put(0,-5){\circle*{0.2}}

\put(4,-4){\circle*{0.2}}

\put(4,4){\circle*{0.2}}

\put(-4,-4){\circle*{0.2}}

\put(-4,4){\circle*{0.2}}

\put(2,-2){\circle*{0.2}}

\put(2,2){\circle*{0.2}}

\put(-2,-2){\circle*{0.2}}

\put(-2,2){\circle*{0.2}}

\put(-4,-4){\line(1,1){8}}

\put(-4,4){\line(1,-1){8}}

\put(-2,-2){\line(0,1){4}}

\put(-2,-2){\line(1,0){4}}

\put(-2,2){\line(1,0){4}}

\put(2,-2){\line(0,1){4}}

\put(-5,0){\line(3,2){3}}

\put(-5,0){\line(3,-2){3}}

\put(5,0){\line(-3,-2){3}}

\put(5,0){\line(-3,2){3}}

\put(0,-5){\line(2,3){2}}

\put(0,-5){\line(-2,3){2}}

\put(0,5){\line(-2,-3){2}}

\put(0,5){\line(2,-3){2}}

\thicklines

\put(-5,0){\line(1,4){1}}

\put(-5,0){\line(1,-4){1}}

\put(5,0){\line(-1,4){1}}

\put(5,0){\line(-1,-4){1}}

\put(0,-5){\line(4,1){4}}

\put(0,-5){\line(-4,1){4}}

\put(0,5){\line(4,-1){4}}

\put(0,5){\line(-4,-1){4}}

\put(5,0){\vector(-1,4){0.6}}

\put(-5,0){\vector(1,-4){.6}}

\put(4,4){\vector(-4,1){2.4}}

\put(4,4){\vector(-4,1){2}}

\put(-4,-4){\vector(4,-1){2.4}}

\put(-4,-4){\vector(4,-1){2}}

\put(0,5){\vector(-4,-1){1.8}}

\put(0,5){\vector(-4,-1){2.2}}

\put(0,5){\vector(-4,-1){2.6}}

\put(0,-5){\vector(4,1){1.8}}

\put(0,-5){\vector(4,1){2.2}}

\put(0,-5){\vector(4,1){2.6}}

\put(-4,4){\vector(-1,-4){0.45}}

\put(-4,4){\vector(-1,-4){0.55}}

\put(-4,4){\vector(-1,-4){0.65}}

\put(-4,4){\vector(-1,-4){0.75}}

\put(4,-4){\vector(1,4){0.45}}

\put(4,-4){\vector(1,4){0.55}}

\put(4,-4){\vector(1,4){0.65}}

\put(4,-4){\vector(1,4){0.75}}

\put(-.2,-.9){0}

\put(-.2,5.4){0}

\put(-.2,-5.9){0}

\put(-5.7,-.3){0}

\put(5.3,-.3){0}

\put(1.8,2.4){1}

\put(1.8,-3){2}

\put(-2.2,2.4){2}

\put(-2.2,-3){1}

\put(-4.5,-4.8){2}

\put(-4.6,4.1){1}

\put(4.1,-4.8){1}

\put(4.2,4.1){2}

\put(-0.25,-7){\textit{b}}

\end{picture}}

\end{picture}
\caption{Triangulations of~$\RP^2$: (\textit{a})~$7$-vertex;
(\textit{b})~$9$-vertex} \label{fig_RP2}

\end{figure}
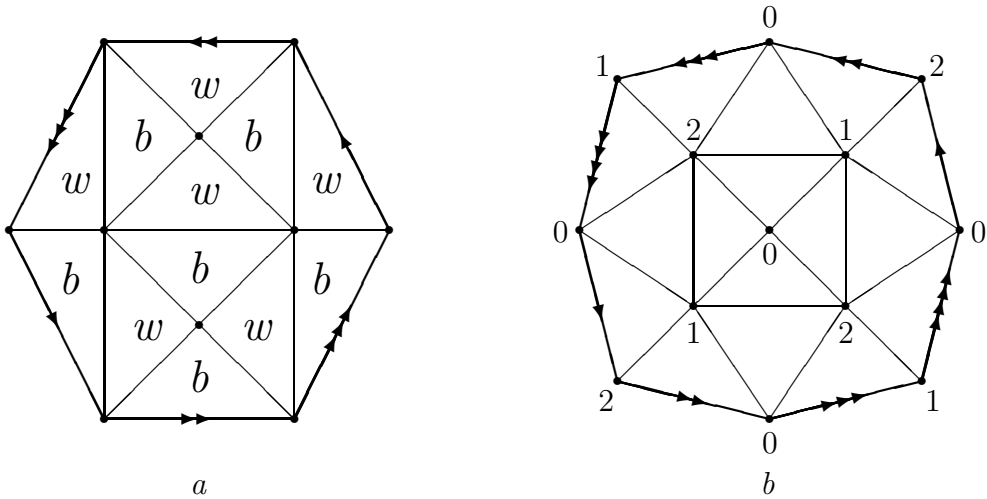

\begin{ex}
The complex projective plane is simply-connected. Therefore,
$\T_{colour}(\CP^2)=\T_{bw}(\CP^2)=\T_{even}(\CP^2)$. By
proposition~\ref{propos_3n+3} a minimal triangulation in the
class~$\T_{colour}(\CP^2)$ has at least $15$ vertices. On the
other hand, the combinatorial manifold~$X$ constructed in
section~~\ref{section_constr} has $15$ vertices and admits a
regular colouring of vertices in $5$ colours. In
section~\ref{section_explicit} we shall prove that $X$ is a
piecewise linear triangulation of~$\CP^2$. Hence $X$ is a minimal
triangulation in the class~$\T_{colour}(\CP^2)$. The author does
not know whether a minimal triangulation in this class is unique.
\end{ex}

\section{Explicit realization of~$X$ as a triangulation of~$\CP^2$}
\label{section_explicit} First, let us construct a representation
of the group~$\Sym(X)=S_4\times S_3$ to the isometry group of the
Fubini--Study metric on~$\CP^2$.

Let $\conj:\bC^3\to\bC^3$ be the operator of coordinatewise
complex conjugation,
$\conj(z_1,z_2,z_3)=(\bar{z}_1,\bar{z}_2,\bar{z}_3)$. By $\hU(3)$
we denote the group of $\R$-linear automorphisms of~$\bC^3$
generated by the unitary group~$U(3)$ and the operator~$\conj$.
Any element of~$\hU(3)$ is either $\bC$-linear or
$\bC$-antilinear. The unitary group $U(3)\subset\hU(3)$ is a
subgroup of index $2$ and coincides with the
intersection~$\hU(3)\cap GL(3,\bC)$. By $D$ we denote the group of
diagonal unitary matrices, $D=\{\lambda E\mid |\lambda|=1\}$. Then
$D\subset\hU(3)$ is a normal subgroup. (Notice that the
subgroup~$D$ is not central.) We put $P\hU(3)=\hU(3)/D$. By~$[g]$
we denote the image of an element~$g\in\hU(3)$ in~$P\hU(3)$. The
projective unitary group~$PU(3)$ is a subgroup of index~$2$
of~$P\hU(3)$. The group $PU(3)$ acts on~$\CP^2$ by isometries of
the Fubini--Study metric. The operator $\conj$ also induces the
isometry $(z_1:z_2:z_3)\mapsto(\bar z_1:\bar z_2:\bar z_3)$.
Therefore the group~$P\hU(3)$ is the subgroup of the isometry
group of the Fubini--Study metric. Indeed, $P\hU(3)$ coincides
with the group of all isometries of~$\CP^2$.

Now we construct a projective representation $R:S_4\times S_3\to
P\hU(3)$ in the following way. We start with the standard
representation $\rho:S_4\to O(3,\R)\subset U(3)$ that realizes the
group~$S_4$ as the symmetry group of a regular
tetrahedron~$T\subset\R^3$ with vertices numbered by~$1,2,3,4$. It
is convenient to us to place the tetrahedron~$T$ in~$\R^3$ so that
the vectors~$e_1,e_2,e_3$ of the standard orthonormal basis
of~$\R^3$ coincide with the vectors from the center of the
tetrahedron~$T$ to the midpoints of the edges $14$, $24$, and $34$
respectively. Then the representation~$\rho$ is given on the
standard generators of~$S_4$ by
$$
\rho\bigl((12)\bigr)=
\begin{pmatrix}
0&1&0\\
1&0&0\\
0&0&1
\end{pmatrix};\qquad \rho\bigl((23)\bigr)=
\begin{pmatrix}
1&0&0\\
0&0&1\\
0&1&0
\end{pmatrix};\qquad
\rho\bigl((34)\bigr)=
\begin{pmatrix}
0&-1&0\\
-1&0&0\\
0&0&1
\end{pmatrix}.
$$
Now we define a representation $\hrho: S_4\to\hU(3)$ by putting
$\hrho(\theta)=\rho(\theta)$ if $\theta$ is an even permutation
and $\hrho(\theta)=\rho(\theta)\conj$ if $\theta$ is an odd
permutation. The representation~$\hrho$ is well defined, since
real matrices commute with~$\conj$ and~$\conj$ has order~$2$.

Now we define a representation $\eta:S_3\to \hU(3)$ on generators
by
$$
\eta\bigl((12)\bigr)=\conj;\qquad \eta\bigl((123)\bigr)=
\begin{pmatrix}
1&0&0\\
0&\omega&0\\
0&0&\omega^2
\end{pmatrix},
$$
where $\omega=e^{\frac{2\pi i}3}$ is the cubic root of unity. It
can be immediately checked that the representation~$\eta$ is well
defined and the following proposition holds.

\begin{propos}
If $\theta\in S_4$ and $\varkappa\in S_3$, then the commutator
$$(\hrho(\theta),\eta(\varkappa))=\hrho(\theta)\eta(\varkappa)\hrho(\theta)^{-1}\eta(\varkappa)^{-1}$$
lies in~$D$.
\end{propos}

This proposition implies that the formula
$$
R\bigl((\theta,\varkappa)\bigr)=\left[\hrho(\theta)\eta(\varkappa)\right]
$$
yields a well-defined projective representation $R:S_4\times
S_3\to P\hU(3)$. In the sequel we always regard the
group~$P\hU(3)$ as an isometry group of the Fubini--Study metric
and identify the group~$S_4\times S_3$ with a subgroup
of~$P\hU(3)$ by the representation~$\eta$.

Recall that~$\CP^2$ can be decomposed into three four-dimensional
disks~$B_1$, $B_2$, and $B_3$ whose pairwise intersections are
solid tori~$\Pi_1$, $\Pi_2$, and $\Pi_3$ and whose triple
intersection is a two-dimensional torus~$T$. A triangulation
of~$\CP^2$ is said to be \textit{equilibrium}\/, if the
disks~$B_j$, the solid tori~$\Pi_j$, and the torus~$T$ are
subcomplexes of this triangulation. T.\,F.\,Banchoff and
W.\,K\"uhnel constructed a series of equilibrium triangulations
of~$\CP^2$ depending on two coprime natural numbers~$p$ and~$q$.
The triangulation corresponding to a pair $(p,q)$ has
$p^2+pq+q^2+3$ vertices. It is based on a $(p^2+pq+q^2)$-vertex
triangulation of the two-dimensional torus~$T$. This triangulation
of the torus~$T$ is the \textit{regular map}\/ $\{3,6\}_{p,q}$ on
the torus (see~\cite{CoMo80}). The triangulation~$\{3,6\}_{p,q}$
can be obtained in the following way. Let us consider the standard
triangulation of plane by rectilinear triangles with vertices in
the points of the lattice generated by two unit vectors~$e_1$ and
$e_2$ with angle~$\frac{2\pi}{3}$ between them. We factorize this
triangulation by the sublattice generated by the vectors with
coordinates $(p-q,2p+q)$ and $(2p+q,p+2q)$ in the
basis~$(e_1,e_2)$. If $p+q\ge 3$, then we obtain a well-defined
$(p^2+pq+q^2)$-vertex triangulation of the two-dimensional torus.
For coprime $p$ and $q$ T.\,F.\,Banchoff and W.\,K\"uhnel found a
way to extend the triangulation~$\{3,6\}_{p,q}$ to a triangulation
of a solid torus without adding new vertices. Since the
triangulation~$\{3,6\}_{p,q}$ is invariant under rotation
by~$\frac{2\pi}{3}$, such extension automatically yields three
different extensions, which are taken to each other by rotations.
Realize these three extensions on the solid tori~$\Pi_1$, $\Pi_2$,
and~$\Pi_3$. Now we introduce three new vertices in the interiors
of the disks~$B_1$, $B_2$, and $B_3$ and triangulate every of
these disks as a cone over the constructed triangulation of its
boundary. The simplest of the triangulations obtained corresponds
to the pair~$(2,1)$ and has $10$ vertices. The corresponding
triangulation~$\{3,6\}_{2,1}$ of the torus is the minimal
triangulation of the torus (see Fig.~\ref{fig_T2},\textit{a}).

Let us now realize the combinatorial manifold~$X$ as an
equilibrium triangulation of~$\CP^2$. Our construction is based on
the $12$-vertex triangulation~$\{3,6\}_{2,2}$ of the
two-dimensional torus. We shall follow Banchoff--K\"uhnel's method
everywhere except for the way of constructing triangulations of
solid tori~$\Pi_j$ because Banchoff--K\"uhnel's method does not
work if the numbers~$p$ and~$q$ are not coprime.

Let us consider the subcomplex~$\T\subset X$ consisting of all
simplices~$\sigma$ such that $\sigma\cup\{\nu\}$ is a simplex
of~$X$ for any vertex~$\nu\in V_4\setminus\{e\}$. It is easy to
check that $\T$ is a triangulation the two-dimensional torus shown
in Fig.~\ref{fig_t12}. This triangulation is isomorphic to the
triangulation~$\{3,6\}_{2,2}$. Now let us consider the
subcomplex~$\P_1\subset X$ consisting of all simplices~$\sigma$
such that $\sigma\cup\{(12)(34)\}$ and $\sigma\cup\{(13)(24)\}$
are simplices of~$X$. (The subcomplexes~$\P_2$ and~$\P_3$ are
defined in a similar way.) The one-dimensional skeleton of~$\P_1$
is obtained from the one-dimensional skeleton of~$\T$ by adding
the $6$ edges $(1,b),(4,b)$ and $(2,b),(3,b)$, $b=1,2,3$. (These
edges are drawn by dotted arcs in Fig.~\ref{fig_t12}.) The
complex~$\P_1$ consists of the three-dimensional simplices
$$
(1,b_1),(4,b_1),(2,b_2),(3,b_3),\quad
(2,b_1),(3,b_1),(1,b_2),(4,b_3),\quad
(1,b_1),(4,b_1),(2,b_2),(3,b_2),
$$
where $b_1,b_2,b_3$ are pairwise distinct, and faces of these
three-dimensional simplices. It is easy to see that $\P_1$ is a
$12$-vertex triangulation of a solid torus with boundary~$\T$. The
triangulation~$\P_1$ is shown in Fig.~\ref{fig_st12}. Here we
imply that the left face and the right face of the parallelepiped
shown in this figure are identified after rotation by~$\pi$.
Similarly, $\P_2$ and $\P_3$ are also triangulated solid tori with
the same vertex sets. Besides, it is easy to show that the
complex~$X$ is obtained from the complex~$\P_1\cup\P_2\cup\P_3$ by
adding cones over the subcomplexes~$\P_1\cup P_2$, $\P_1\cup
\P_3$, and~$\P_2\cup\P_3$ with vertices~$(12)(34)$, $(13)(24)$,
and~$(14)(23)$ respectively. Thus the complex~$X$ can be realized
as an equilibrium triangulation of~$\CP^2$ in the following way.
Firstly, we realize the subcomplex~$\T$ as a triangulation of the
torus~$T$. Secondly, we realize the subcomplexes~$\P_j$ as
triangulations of solid tori~$\Pi_j$, $j=1,2,3$. Finally, we put
vertices~$(14)(23)$, $(13)(24)$, and~$(12)(34)$ inside the
four-dimensional disks~$B_1$, $B_2$, and $B_3$ respectively and
add to the triangulation the cones with these vertices over the
triangulations~$\P_2\cup\P_3$, $\P_1\cup \P_3$, and~$\P_1\cup P_2$
respectively.

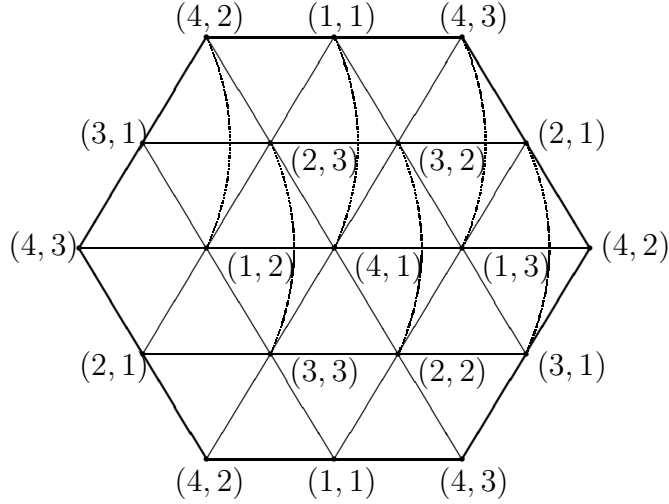
\begin{figure}
\unitlength=0.28cm
\begin{picture}(30,24)
\put(7.5,0.5){$(4,2)$}

\put(13.7,0.5){$(1,1)$}

\put(19.9,0.5){$(4,3)$}

\put(7.5,22.5){$(4,2)$}

\put(13.7,22.5){$(1,1)$}

\put(19.9,22.5){$(4,3)$}

\put(3,6){$(2,1)$}

\put(3,17){$(3,1)$}

\put(-0.3,11.7){$(4,3)$}

\put(24.5,6){$(3,1)$}

\put(24.5,17){$(2,1)$}

\put(27.5,11.7){$(4,2)$}

\put(12.9,5.7){$(3,3)$}

\put(18.9,5.7){$(2,2)$}

\put(9.9,10.7){$(1,2)$}

\put(15.9,10.7){$(4,1)$}

\put(21.9,10.7){$(1,3)$}

\put(18.9,15.7){$(3,2)$}

\put(12.9,15.7){$(2,3)$}

\multiput(9,22)(6,0){3}{\circle*{.3}}

\multiput(6,17)(6,0){4}{\circle*{.3}}

\multiput(3,12)(6,0){5}{\circle*{.3}}

\multiput(6,7)(6,0){4}{\circle*{.3}}

\multiput(9,2)(6,0){3}{\circle*{.3}}

\put(6,17){\line(1,0){18}}

\put(3,12){\line(1,0){24}}

\put(6,7){\line(1,0){18}}

\put(6,7){\line(3,5){9}}

\put(9,2){\line(3,5){12}}

\put(15,2){\line(3,5){9}}

\put(6,17){\line(3,-5){9}}

\put(9,22){\line(3,-5){12}}

\put(15,22){\line(3,-5){9}}

\qbezier[80](9,12)(11.2,17)(9,22)

\qbezier[80](15,12)(17.2,17)(15,22)

\qbezier[80](21,12)(23.2,17)(21,22)

\qbezier[80](12,7)(14.2,12)(12,17)

\qbezier[80](18,7)(20.2,12)(18,17)

\qbezier[80](24,7)(26.2,12)(24,17)

\thicklines

\put(9,22){\line(1,0){12}}

\put(9,2){\line(1,0){12}}

\put(3,12){\line(3,5){6}}

\put(21,2){\line(3,5){6}}

\put(3,12){\line(3,-5){6}}

\put(21,22){\line(3,-5){6}}

\end{picture}
\caption{The $12$-vertex triangulation~$\T$ of the two-dimensional
torus} \label{fig_t12}
\end{figure}
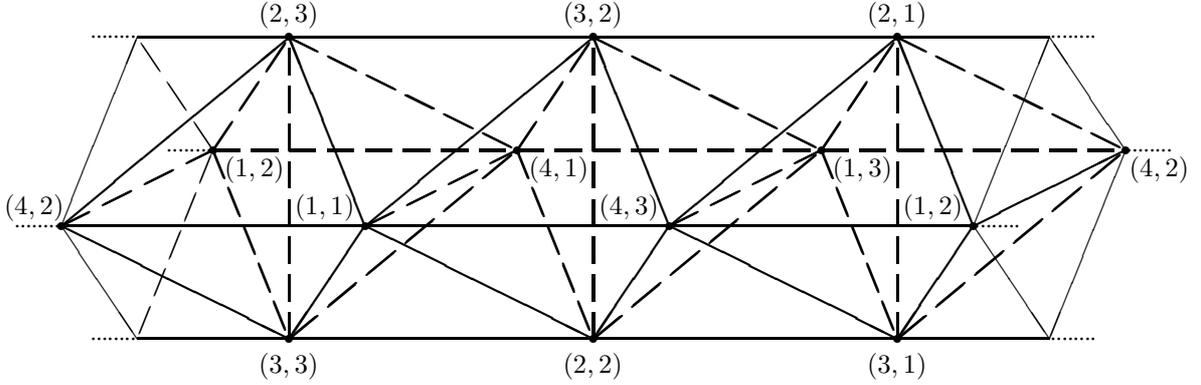
\begin{figure}
\unitlength=5mm
\begin{picture}(30,11)

\thicklines

\put(1,4.5){\line(1,0){24}}

\put(25,4.5){\circle*{0.2}}

\put(29,6.5){\circle*{0.2}}

\put(3,1.5){\line(1,0){24}}

\put(3,9.5){\line(1,0){24}}

\multiput(5,6.5)(1,0){24}{\line(1,0){0.7}}

\multiput(1,1.5)(8,0){3}{\begin{picture}(0,0)

\put(0,3){\circle*{0.2}}

\put(4,5){\circle*{0.2}}

\put(6,0){\circle*{0.2}}

\put(6,8){\circle*{0.2}}

\put(0,3){\line(6,5){6}}

\put(0,3){\line(2,-1){6}}

\put(8,3){\line(-2,5){2}}

\put(8,3){\line(-2,-3){2}}

\multiput(0,3)(1.1,0.55){4}{\line(2,1){0.75}}

\multiput(4,5)(0.7,1.05){3}{\line(2,3){0.5}}

\multiput(4,5)(0.4,-1){5}{\line(2,-5){0.3}}

\multiput(6,0)(0,1){8}{\line(0,1){0.7}}

\multiput(6,0)(1.02,0.85){6}{\line(6,5){0.75}}

\multiput(6,8)(1.02,-0.51){6}{\line(2,-1){0.75}}

\end{picture}}

\put(25,4.5){\line(2,1){4}}

\thinlines

\put(1,4.5){\line(2,5){2}}

\put(1,4.5){\line(2,-3){2}}

\put(25,4.5){\line(2,5){2}}

\put(25,4.5){\line(2,-3){2}}

\put(29,6.5){\line(-2,-5){2}}

\put(29,6.5){\line(-2,3){2}}

\multiput(4.95,6.375)(-0.4,-1){5}{\line(-2,-5){0.3}}

\multiput(5,6.5)(-0.7,1.05){3}{\line(-2,3){0.5}}

\multiput(0.9,4.5)(-0.15,0){8}{\circle*{0.03}}

\multiput(2.9,1.5)(-0.15,0){8}{\circle*{0.03}}

\multiput(4.9,6.5)(-0.15,0){8}{\circle*{0.03}}

\multiput(2.9,9.5)(-0.15,0){8}{\circle*{0.03}}

\multiput(25.1,4.5)(0.15,0){8}{\circle*{0.03}}

\multiput(27.1,1.5)(0.15,0){8}{\circle*{0.03}}

\multiput(29.1,6.5)(0.15,0){8}{\circle*{0.03}}

\multiput(27.1,9.5)(0.15,0){8}{\circle*{0.03}}

\footnotesize

\put(7.15,4.8){$(1,1)$}

\put(15.15,4.8){$(4,3)$}

\put(23.15,4.8){$(1,2)$}

\put(-0.5,4.8){$(4,2)$}

\put(5.3,5.8){$(1,2)$}

\put(13.3,5.8){$(4,1)$}

\put(21.3,5.8){$(1,3)$}

\put(29.1,5.8){$(4,2)$}

\put(6.2,9.9){$(2,3)$}

\put(14.2,9.9){$(3,2)$}

\put(22.2,9.9){$(2,1)$}

\put(6.2,0.6){$(3,3)$}

\put(14.2,0.6){$(2,2)$}

\put(22.2,0.6){$(3,1)$}

\end{picture}

\caption{The $12$-vertex triangulation~$\P_1$ of the solid torus}
\label{fig_st12}
\end{figure}

Now we shall construct these realizations explicitly so that the
action of the automorphism group~$\Sym(X)$ on~$X$ will be
identified with the constructed action of the group~$S_4\times
S_3$ by isometries of~$\CP^2$. First, to vertices $s$ of~$X$ we
assign the following points~$v_s\in\CP^2$,
\begin{gather*}
v_{(14)(23)}=(1:0:0),\qquad v_{(13)(24)}=(0:1:0),\qquad
v_{(12)(34)}=(0:0:1),\\
v_{(1,b)}=(-1:\omega^b:\omega^{2b}),\qquad
v_{(2,b)}=(1:-\omega^b:\omega^{2b}),\qquad
v_{(3,b)}=(1:\omega^b:-\omega^{2b}),\\
v_{(4,b)}=(1:\omega^b:\omega^{2b}),\qquad b=1,2,3.
\end{gather*}
It can be immediately checked that this set of points is
$(S_4\times S_3)$-invariant and $R(h)v_s=v_{h\cdot s}$ for any
element~$h\in S_4\times S_3$ and any vertex~$s$ of~$X$.

The points $v_{(a,b)}$ lie in the torus~$T$. The torus~$T$ is flat
in the Fubini--Study metric and the $12$ points
$(1:\pm\omega^b:\pm\omega^{2b})$ lie in it exactly at vertices of
the triangulation~$\{3,6\}_{2,2}$ consisting of rectilinear
triangles. Thus the triangulation~$\T$ is realized as a flat
triangulation of the torus~$T$ with vertices at the
points~$v_{(a,b)}$. The edges of~$\T$ are realized by the shortest
geodesic segments. For example, the edge with endpoints $(1,1)$
and $(2,2)$ is realized by the geodesic segment
$\left(e^{-\frac{i\pi}3(1+t)}:e^{-\frac{i\pi}3(2-t)}:1\right)$,
$t\in[0,1]$. Other edges can be obtained from this one by the
action of the group~$S_4\times S_3$.

Let us now construct an explicit triangulation of the solid
torus~$\Pi_1$. We realize the edge~$(1,1),(4,1)$ by the geodesic
segment~$(-t:\omega:\omega^2)$, $t\in[-1,1]$. Other
edges~$(a_1,b),(a_2,b)$ can be obtained from this one by the
action of the group~$S_4\times S_3$. To realize three-dimensional
simplices it is convenient to endow the solid torus~$\Pi_1$ by a
metric different from the Fubini--Study metric, namely, by a flat
metric coinciding with the Fubini--Study metric on the torus~$T$.
This flat metric is defined in the following way. We parametrize
the torus~$T$ by
$$
\bt(\varphi,\psi)=\left(e^{i\varphi}:e^{i\psi}:e^{-i\psi}\right),\qquad
\varphi,\psi\in\R/(2\pi\Z).
$$
This parametrization is two-to-one
since~$\bt(\varphi+\pi,\psi+\pi)=\bt(\varphi,\psi)$. The
restriction of the Fubini--Study metric to the torus~$T$ is equal
to~$\frac{2}{9}d\varphi^2+\frac{2}{3}d\psi^2$. The length of the
circle~$\psi=c$, which is homological to zero in the solid
torus~$\Pi_1$, is equal to~$\frac{2\sqrt{2}\pi}{3}$. Now we
parametrize (two-to-one) the solid torus~$\Pi_1$ by
\begin{multline*}
\bp(x,y,h)=\left(2(|x|+|y|)\left( \sin\frac{\pi
x}{2(|x|+|y|)}+i\sin\frac{\pi y}{2(|x|+|y|)}\right)\right.:\\
\left.\vphantom{\frac{\pi
y}{2(|x|+|y|)}}:\left(\frac{\pi}{2}-|x|-|y|\right)e^{i\sqrt{\frac32}h}:
\left(\frac{\pi}{2}-|x|-|y|\right)e^{-i\sqrt{\frac32}h}\right),\quad
|x|+|y|\le \frac{\pi}{6},\
h\in\R/{\textstyle\left(2\sqrt{\frac23}\pi\Z\right)}
\end{multline*}
and endow it with the flat metric~$dx^2+dy^2+dh^2$. It is easy to
check that the restriction of this metric to the torus~$T$
coincides with the Fubini--Study metric. Indeed, every
section~$h=c$ of the solid torus is the square with
vertices~$\bigl(\pm\frac{\pi}6,0\bigr),\bigl(0,\pm\frac{\pi}6\bigr)$
and perimeter~$\frac{2\sqrt{2}\pi}{3}$ in the Euclidean
coordinates~$(x,y)$.

The coordinates $(x,y,h)$ yield an isometric embedding of the
universal covering of the solid torus~$\Pi_1$ into~$\R^3$. The
image of this embedding is an infinite cylinder~$C$ over a square
with side length~$\frac{\sqrt{2}\pi}{6}$. The cylinder~$C$ is
given by the inequality~$|x|+|y|\le\frac{\pi}6$. The solid
torus~$\Pi_1$ is the quotient of the cylinder~$C$ by the action of
the infinite cyclic group~$\langle S\rangle$ generated by the
isometry~$S:(x,y,h)\mapsto\bigl(-x,-y,h+\sqrt{\frac23}\pi\bigr)$.
Let $\widetilde{\P}_1$ be the universal covering of~$\P_1$. Under
the constructed embedding the vertices of~$\widetilde{\P}_1$ go to
the points $\left(\pm\frac{\pi}{6},0,\sqrt{\frac23}\frac{\pi
k}{3}\right)$,
$\left(0,\pm\frac{\pi}{6},\sqrt{\frac23}\bigl(\frac{\pi}{6}+\frac{\pi
k}{3}\bigr)\right)$, $k\in\Z$, and the edges of~$\widetilde{\P}_1$
go to rectilinear segments. Let us decompose the cylinder~$C$ into
convex tetrahedra as it is shown on Fig.~\ref{fig_st12}. This
decomposition consists of three series of tetrahedra

1) $\left(\frac{\varepsilon\pi}{6},0,\sqrt{\frac23}\frac{\pi
k}{3}\right),\left(0,\pm\frac{\pi}{6},\sqrt{\frac23}\left(\frac{\pi}{6}+\frac{\pi
k}{3}\right)\right),\left(\frac{\varepsilon\pi}{6},0,\sqrt{\frac23}
\frac{\pi(k+1)}{3}\right)$, where~$k\in\Z$, $\varepsilon=\pm 1$;

2)
$\left(0,\frac{\varepsilon\pi}{6},\sqrt{\frac23}\left(-\frac{\pi}{6}+\frac{\pi
k}{3}\right)\right),\left(\pm\frac{\pi}{6},0,\sqrt{\frac23}\frac{\pi
k}{3}\right),\left(0,\frac{\varepsilon\pi}{6},\sqrt{\frac23}\left(\frac{\pi}{6}+\frac{\pi
k}{3}\right)\right)$, where~$k\in\Z$, $\varepsilon=\pm 1$;

3) $\left(\pm\frac{\pi}{6},0,\sqrt{\frac23}\frac{\pi
k}{3}\right),\left(0,\pm\frac{\pi}{6},\sqrt{\frac23}\left(\frac{\pi}{6}+\frac{\pi
k}{3}\right)\right)$, where $k\in \Z$.

The triangulation constructed is invariant under the isometry~$S$
and, hence, yields a triangulation of the solid torus~$\Pi_1$. The
latter triangulation is the required realization of the simplicial
complex~$\P_1$. The triangulations of the solid tori~$\Pi_2$
and~$\Pi_3$ are constructed similarly.

Now we need to construct explicitly triangulations of the
disks~$B_j$. We consider the
coordinates~$Z_2=\frac{z_2|z_1|}{z_1(|z_1|+|z_2|+|z_3|)}$ and
$Z_3=\frac{z_3|z_1|}{z_1(|z_1|+|z_2|+|z_3|)}$. In these
coordinates the set~$B_1$ is given by the inequalities
$2|Z_1|+|Z_2|\le 1$, $2|Z_2|+|Z_1|\le 1$. The boundary of this
convex set is the union of the solid tori~$\Pi_2$ and~$\Pi_3$ with
triangulations constructed above. We triangulate the disk~$B_1$ as
the affine cone over the constructed triangulation of the
union~$\Pi_2\cup\Pi_3$ with vertex~$v_{(14)(23)}=(0,0)$.
Similarly, we triangulate the disks~$B_2$ and~$B_3$.

The above construction allows us to obtain an explicit formula for
the $(S_4\times S_3)$-equivariant piecewise smooth
homeomorphism~$f:|X|\to\CP^2$. We shall give explicit formulae for
the restrictions of the homeomorphism~$f$ to the four-dimensional
simplices
\begin{equation} \label{eq_Delta}
\begin{aligned}
\Delta_1&=\{(12)(34),(1,3),(2,1),(3,2),(4,3)\},\\
\Delta_2&=\{(12)(34),(1,3),(2,1),(3,1),(4,3)\}.
\end{aligned}
\end{equation}
The restrictions of~$f$ to other four-dimensional simplices of~$X$
can be obtained from the restrictions to the simplices~$\Delta_1$
and $\Delta_2$ by the action of the group~$S_4\times S_3$. The
restrictions of~$f$ to the simplices~$\Delta_1$ and~$\Delta_2$ are
given by
\begin{multline}
\label{eq_emb1} (\xi_0,\xi_1,\ldots,\xi_4)\mapsto
\left(\frac{|\xi_4-\xi_1|+\xi_2+\xi_3}{3}
\left(\sin\frac{\pi(\xi_4-\xi_1)}
{2(|\xi_4-\xi_1|+\xi_2+\xi_3)}+\right.\right.\\
\left.+i\sin\frac{\pi(\xi_2+\xi_3)}
{2(|\xi_4-\xi_1|+\xi_2+\xi_3)}\right):
\left(\frac{1-\xi_0}{2}-\frac{|\xi_4-\xi_1|+\xi_2+\xi_3}{6}\right)e^{\frac{i\pi(\xi_2-\xi_3)}{6(1-\xi_0)}}:\\
:\left.\left(\frac{1+\xi_0}{2}-\frac{|\xi_4-\xi_1|+\xi_2+\xi_3}{6}\right)e^{-\frac{i\pi(\xi_2-\xi_3)}{6(1-\xi_0)}}\right),
\end{multline}
\begin{multline}\label{eq_emb2}
(\zeta_0,\zeta_1,\ldots,\zeta_4)\mapsto
\left(\frac{|\zeta_4-\zeta_1|+|\zeta_2-\zeta_3|}{3}
\left(\sin\frac{\pi(\zeta_4-\zeta_1)}
{2(|\zeta_4-\zeta_1|+|\zeta_2-\zeta_3|)}+\right.\right.\\
\left.+i\sin\frac{\pi(\zeta_2-\zeta_3)}{2(|\zeta_4-\zeta_1|+|\zeta_2-\zeta_3|)}\right):
\left(\frac{1-\zeta_0}{2}-\frac{|\zeta_4-\zeta_1|+|\zeta_2-\zeta_3|}{6}\right)e^{\frac{i\pi(\zeta_2+\zeta_3)}{6(1-\zeta_0)}}:\\
:\left.\left(\frac{1+\zeta_0}{2}-\frac{|\zeta_4-\zeta_1|+|\zeta_2-\zeta_3|}{6}\right)e^{-\frac{i\pi(\zeta_2+\zeta_3)}{6(1-\zeta_0)}}\right),
\end{multline}
where $(\xi_0,\xi_1,\ldots,\xi_4)$ and
$(\zeta_0,\zeta_1,\ldots,\zeta_4)$,
$\sum_{j=0}^4\xi_j=\sum_{j=0}^4\zeta_j=1$, are the barycentric
coordinates in the simplices~$\Delta_1$ and~$\Delta_2$
respectively. (The barycentric coordinates are numbered
corresponding to those orderings of vertices of
simplices~$\Delta_1$ and $\Delta_2$ which are given
by~(\ref{eq_Delta}).)

In section~\ref{section_constr} we defined the subdivision~$\bX$
of the simplicial complex~$X$. It is convenient to rewrite
formulae~(\ref{eq_emb1}) and~(\ref{eq_emb2}) for the
homeomorphism~$f$ in the barycentric coordinates for simplices
of~$\bX$. The vertices~$(\widehat{a_1a_2},b)$ of~$\bX$ are the
midpoints of those edges of~$X$ that are contained in the
subcomplex~$\P_1\cup\P_2\cup\P_3$ but are not contained in the
subcomplex~$\T$. Hence under the homeomorphism~$f$ these vertices
go to the points
\begin{align*}
v_{(\widehat{12},b)}&=(1:-\omega^b:0),&
v_{(\widehat{13},b)}&=(-\omega^b:0:1),&
v_{(\widehat{23},b)}&=(0:1:-\omega^b),\\
v_{(\widehat{34},b)}&=(1:\omega^b:0),&
v_{(\widehat{24},b)}&=(\omega^b:0:1),&
v_{(\widehat{14},b)}&=(0:1:\omega^b).
\end{align*}
Recall that four-dimensional simplices of~$\bX$ are divided into
two $(S_4\times S_3)$-orbits with representatives
\begin{equation}
\label{eq_sigma}
\begin{aligned}
\sigma_1&=\{(12)(34),(\widehat{14},3),(2,1),(3,2),(4,3)\},\\
\sigma_2&=\{(12)(34),(\widehat{14},3),(\widehat{23},1),(2,1),(4,3)\}.
\end{aligned}
\end{equation}
Let $(p_0,p_1,\ldots,p_4)$ be the barycentric coordinates in the
simplex~$\sigma_1$. (The vertices of~$\sigma_1$ are ordered as
they are listed in~(\ref{eq_sigma}).) The simplex~$\sigma_1$ is
contained in the simplex~$\Delta_1$ and we have~$\xi_0=p_0$,
$\xi_1=\frac{p_1}{2}$, $\xi_2=p_2$, $\xi_3=p_3$,
$\xi_4=p_4+\frac{p_1}{2}$. Thus formula~(\ref{eq_emb1}) implies
that the restriction of the homeomorphism~$f$ to the
simplex~$\sigma_1$ is given by
\begin{multline}\label{eq_emb1'}
(p_0,p_1,\ldots,p_4)\mapsto
\left(\frac{p_2+p_3+p_4}{3}e^{\frac{i\pi
(p_2+p_3)}{2(p_2+p_3+p_4)}}:
\left(\frac{p_1}{2}+\frac{p_2+p_3+p_4}{3}\right)
e^{\frac{i\pi(p_2-p_3)}{6(1-p_0)}}:\right.\\
:\left.\left(p_0+\frac{p_1}{2}+\frac{p_2+p_3+p_4}{3}\right)
e^{-\frac{i\pi(p_2-p_3)}{6(1-p_0)}}\right).
\end{multline}
Similarly, let $(q_0,q_1,\ldots,q_4)$ be the barycentric
coordinates in the simplex~$\sigma_2$. (The vertices of~$\sigma_2$
are ordered as they are listed in~(\ref{eq_sigma}).) Then the
restriction of the homeomorphism~$f$ to the simplex~$\sigma_2$ is
given by
\begin{multline}\label{eq_emb2'}
(q_0,q_1,\ldots,q_4)\mapsto \left(\frac{q_3+q_4}{3}e^{\frac{i\pi
q_3}{2(q_3+q_4)}}:
\left(\frac{q_1+q_2}{2}+\frac{q_3+q_4}{3}\right)
e^{\frac{i\pi(q_2+q_3)}{6(1-q_0)}}:\right.\\
:\left.\left(q_0+\frac{q_1+q_2}{2}+\frac{q_3+q_4}{3}\right)
e^{-\frac{i\pi(q_2+q_3)}{6(1-q_0)}}\right).
\end{multline}
The formulae for the restrictions of~$f$ to other four-dimensional
simplices of~$\bX$ can be obtained from formulae~(\ref{eq_emb1'})
and~(\ref{eq_emb2'}) by the action of the group~$S_4\times S_3$.

\section{Triangulation of the moment mapping}
\label{section_moment} In this section we shall construct a
triangulation of the classical moment
mapping~$\mu:\CP^2\to\Delta^2$ given by~(\ref{eq_moment}). This
means that we shall construct a triangulation of~$\CP^2$ and a
triangulation of the triangle~$\Delta^2$ such that the mapping
$\mu$ is simplicial with respect to this pair of triangulations.
For a triangulation of~$\CP^2$ we take the triangulation~$\bX$.
(Recall that the explicit homomorphism~$f:|\bX|\to\CP^2$ was
constructed in section~\ref{section_explicit}).) For a
triangulation of the triangle~$\Delta^2$ we take its barycentric
subdivision~$(\Delta^2)'$.

We introduce an action of the group~$S_4\times S_3$ on the
triangle~$\Delta^2$ such that the multiplier~$S_3$ acts trivially
and the multiplier~$S_4$ acts by linear mappings permuting the
vertices of the triangle~$\Delta^2$ according to the
homomorphism~$S_4\to S_4/V_4\cong S_3$. Then the action is given
by
\begin{gather*}
(12)\cdot(t_1,t_2,t_3)=(34)\cdot(t_1,t_2,t_3)=(t_2,t_1,t_3);\\
(23)\cdot(t_1,t_2,t_3)=(t_1,t_3,t_2)
\end{gather*}
on the generators of~$S_4$. Recall that the group~$S_4\times S_3$
acts on~$\CP^2$ by isometries and the action is given by the
projective representation~$R$ (see
section~\ref{section_explicit}). It is easy to check that the
mapping~$\mu$ is equivariant with respect to the above pair of
actions.

We define a simplicial mapping $m:\bX\to (\Delta^2)'$ on the
vertices of~$\bX$ by
\begin{gather*}
m\bigl((12)(34)\bigr)=(0,0,1);\quad
m\bigl((13)(24)\bigr)=(0,1,0);\quad
m\bigl((14)(23)\bigr)=(1,0,0);\\
m\bigl((\widehat{12},b)\bigr)=m\bigl((\widehat{34},b)\bigr)=\left(\frac12,\frac12,0\right);\qquad
m\bigl((\widehat{13},b)\bigr)=m\bigl((\widehat{24},b)\bigr)=\left(\frac12,0,\frac12\right);\\
m\bigl((\widehat{23},b)\bigr)=m\bigl((\widehat{14},b)\bigr)=\left(0,\frac12,\frac12\right);\qquad
m\bigl((a,b)\bigr)=\left(\frac13,\frac13,\frac13\right).
\end{gather*}
It can be immediately checked that the simplicial mapping~$m$ is
well defined and $(S_4\times S_3)$-equivariant. The mapping of the
geometric realizations of the complexes~$\bX$ and~$(\Delta^2)'$
induced by the mapping~$m$ will also be denoted by~$m$.

Ahead with the mapping~$\mu$, it is convenient to consider the
mapping~$\tilde{\mu}:\CP^2\to\Delta^2$ given by
$$
\tilde{\mu}(z_1:z_2:z_3)=\frac{\bigl(|z_1|,|z_2|,|z_3|\bigr)}{|z_1|+|z_2|+|z_3|}.
$$
We have~$\tilde{\mu}=\mu\circ g$, where $g:\CP^2\to\CP^2$ is the
homeomorphism given by
$$
g(z_1:z_2:z_3)=\left(\frac{z_1}{\sqrt{|z_1|}}:\frac{z_2}{\sqrt{|z_2|}}:
\frac{z_3}{\sqrt{|z_3|}}\right).
$$
Here one should replace~$\frac{z_j}{\sqrt{|z_j|}}$ by~$0$
if~$z_j=0$. Obviously, the mappings~$\tilde{\mu}$ and~$g$ are
equivariant.

\begin{propos}
The simplicial mapping~$m$ triangulates the mapping~$\tilde{\mu}$,
that is, there is a commutative diagram
$$
\begin{CD}
|\bX| @>f>> \CP^2\\
@VmVV @V{\tilde{\mu}}VV\\
|(\Delta^2)'| @= \Delta^2
\end{CD}
$$
\end{propos}
\begin{proof}
Both mappings~$m$ and~$\tilde{\mu}\circ f$ are $(S_4\times
S_3)$-equivariant. Hence we suffices to check that they coincide
on two four-dimensional simplices of~$\bX$ representing different
$(S_4\times S_3)$-orbits, for example, on the simplices~$\sigma_1$
and~$\sigma_2$ (see~(\ref{eq_sigma})). The coincidence of the
mappings~$m$ and~~$\tilde{\mu}\circ f$ on the simplices~$\sigma_1$
and~$\sigma_2$ follows easily from formulae~(\ref{eq_emb1'})
and~(\ref{eq_emb2'}).
\end{proof}

Now to obtain a triangulation of the moment mapping~$\mu$ one
should replace the homeomorphism~$f:|\bX|\to\CP^2$ by the
homeomorphism~$g\circ f$.

The preimage of the barycenter of the triangle~$\Delta^2$ under
the mapping~$m$ is the subcomplex~$\T\subset\bX$ isomorphic to the
$12$-vertex triangulation of the two-dimensional torus shown in
Fig.~\ref{fig_t12}. The preimage of the midpoint of every edge of
the triangle~$\Delta^2$ is a subcomplex of~$\bX$ isomorphic to the
boundary of the hexagon. The preimage of every vertex
of~$\Delta^2$ is a vertex of~$\bX$.

\section{Relationship with complex crystallographic groups}
\label{section_crystal}

In this section we shall conveniently regard the group~$S_3$ as
the group of permutations of the set~$\Z_3=\{0,1,2\}$ rather than
the set~$\{1,2,3\}$.

Recall that a \textit{crystallographic group} is a cocompact
discrete group of isometries of a finite-dimensional Euclidean
space, that is, a cocompact subgroup of the semidirect
product~$\R^n\leftthreetimes O(n)$. Similarly, a \textit{complex
crystallographic group} is a cocompact discrete subgroup of the
group~$\bC^n\leftthreetimes U(n)$, which is the group of those
transformations of a finite-dimensional Hermitian space that are
compositions of unitary transformations and translations.

The relationship of K\"uhnel's $9$-vertex triangulation of~$\CP^2$
with complex crystallographic groups was discovered by B.\,Morin
and M.\,Yoshida~\cite{MoYo91} (see also~\cite{ArMa91}). In this
section we recall some results of the
papers~\cite{MoYo91},~\cite{ArMa91} and explain the relationship
between the constructed $15$-vertex triangulation~$X$ with complex
crystallographic groups.

Let $\tau$ be a complex number with positive imaginary part.
By~$L=L(\tau)$ we denote the lattice in~$\bC$ with
basis~$(1,\tau)$. We consider the complex torus~$T^2=\bC/L$, the
$6$-dimensional torus~$T^6=T^2\times T^2\times T^2$, and the
subtorus
$$
T^4=\{(z_1,z_2,z_3)\in T^2\times T^2\times T^2\mid z_1+z_2+z_3=0\}
$$
of the torus~$T^6$. The group $S_3$ acts on~$T^6$ by permutations
of the multipliers~$T^2$. The torus~$T^4$ is invariant under this
action. The following proposition is well known.
\begin{propos}
The quotient~$T^4/S_3$ is homeomorphic to~$\CP^2$.
\end{propos}
\begin{proof}
The complex torus~$T^2$ can be realized as an elliptic
curve~$E\subset\CP^2$. Three points ~$z_1,z_2,z_3\in T^2$ satisfy
the equality~$z_1+z_2+z_3=0$ if and only if the corresponding
three points of~$E$ lie in a complex line. Thus we obtain a
homeomorphism between the space~$T^4/S_3$, which is the space of
unordered triples of points~$z_1,z_2,z_3\in T^2$ such that
$z_1+z_2+z_3=0$, and the space of complex lines in~$\CP^2$. The
latter space is obviously homeomorphic to~$\CP^2$.
\end{proof}

This proposition immediately implies that $\CP^2$ is the quotient
of~$\bC^2$ by a complex crystallographic group. Let us describe
this group explicitly. Consider the two-dimensional subspace
$$
W=\{(z_1,z_2,z_3)\in \bC^3\mid z_1+z_2+z_3=0\}\subset\bC^3
$$
and the vectors~$h_0=(1,-1,0)$, $h_1=(0,1,-1)$, and $h_2=(-1,0,1)$
belonging to it. Obviously, $h_0+h_1+h_2=0$. Let us identify the
group~$S_3$ with the group of transformations of~$\bC^3$ that
permutes the coordinates. The subspace~$W$ is invariant under this
group. Consider the lattice~$\Lambda=Lh_0+Lh_1+Lh_2\subset W$ and
the complex crystallographic group~$\Gamma=\Lambda\leftthreetimes
S_3$. Then~$W/\Gamma=T^4/S_3\approx\CP^2$.

Suppose~$f_0=\frac{h_2-h_1}{3}$, $f_1=\frac{h_0-h_2}{3}$, and
$f_2=\frac{h_1-h_0}{3}$; then $h_0=f_1-f_2$, $h_1=f_2-f_0$, and
$h_2=f_0-f_1$. The lattice~$\Lambda_{\R}=\sum_{j=0}^2\Z h_j$ is
the subgroup of index~$3$  in the
lattice~$\overline{\Lambda}_{\R}=\sum_{j=0}^2\Z f_j$ and the
lattice~$\Lambda$ is the subgroup of index~$9$ in the
lattice~$\overline{\Lambda}=\sum_{j=0}^2Lf_j$. Decompose the
space~$W$ into the direct sum~$W=W_{\R}\oplus\tau W_{\R}$,
where~$W_{\R}$ is the two-dimensional real subspace spanned by the
vectors~$h_0,h_1,h_2$. Similarly, we have the decompositions
$\Lambda=\Lambda_{\R}\oplus\tau\Lambda_{\R}$,
$\overline{\Lambda}=\overline{\Lambda}_{\R}\oplus\tau\overline{\Lambda}_{\R}$.
The lattice~$\overline{\Lambda}_{\R}$ is hexagonal.
By~$\widehat{Z}$ we denote the corresponding triangulation of the
plane~$W_{\R}$ by rectilinear triangles. The direct (though not
orthogonal) product of the triangulations~$\widehat{Z}$ and~$\tau
\widehat{Z}$ is a decomposition of the space~$W$ into convex
prisms~$\Delta^2\times\Delta^2$. This decomposition, which will be
denoted by~$\widehat{Q}$, is invariant under the action
of~$\Gamma$ and $Q=\widehat{Q}/\Gamma$ is a decomposition
of~$\CP^2$ into prisms~$\Delta^2\times\Delta^2$. (Later we shall
see that two prisms in the latter decomposition can possess
several common facets.) B.\,Morin and M.\,Yoshida~\cite{MoYo91}
constructed a $\Gamma$-invariant rectilinear
triangulation~$\widehat{K}$ of the space~$W\approx\bC^2$ such that
$K=\widehat{K}/\Gamma$ is a well-defined triangulation of~$\CP^2$
isomorphic to K\"uhnel's $9$-vertex triangulation. P.\,Arnoux and
A.\,Marin~\cite{ArMa91} noticed that the
triangulation~$\widehat{K}$ is a subdivision of the
decomposition~$\widehat{Q}$ such that every
prism~$\Delta^2\times\Delta^2$ is divided into~$6$
four-dimensional simplices without adding new vertices.

The case~$\tau=\omega=e^{\frac{2i\pi}3}$ is of a special interest
since for $\tau=\omega$ the triangulation~$\widehat{K}$ becomes
invariant under a bigger crystallographic
group~$\overline{\Gamma}\supset\Gamma$ (see~\cite{MoYo91}). The
group~$\overline{\Gamma}$ is the semidirect
product~$\overline{\Lambda}\leftthreetimes G_{18}$, where
$G_{18}\subset U(2)$ is the group of order~$18$ generated by the
subgroup~$S_3\subset U(2)$ and the operator of multiplication
by~$\omega$. The quotient group~$\overline{\Gamma}/\Gamma$ has
order~$27$ and is a subgroup of index~$2$ of the automorphism
group of~$K$.

Let us now describe how the triangulation~$X$ appears in the
context of complex crystallographic groups. The multiplication
by~$\omega$ will not be important in this construction. Hence we
can again suppose that $\tau$ is an arbitrary complex number with
positive imaginary part.

First, let us consider the
decomposition~$\widetilde{Q}=\widehat{Q}/\Lambda$ of the
torus~$T^4$. This decomposition is the direct product of the two
identical decompositions~$\widetilde{Z}=\widehat{Z}/\Lambda_{\R}$
of the two-dimensional torus~$W_{\R}/\Lambda_{\R}$ into triangles.
The decomposition~$\widetilde{Z}$ is shown in Fig.~\ref{fig_6}.
(We imply the identification of opposite sides of the hexagon.)
This decomposition is not a triangulation since all $6$ triangles
of this decomposition have the same $3$ vertices. The group~$S_3$
acts on the torus~$W_{\R}/\Lambda_{\R}$ so that the transpositions
$\sigma_0=(12)$, $\sigma_1=(20)$, and $\sigma_2=(01)$ act by
symmetries in the lines ~$\ell_0$, $\ell_1$, and $\ell_2$
respectively. The group~$S_3$ acts transitively on the set of
triangles of the decomposition~$\widetilde{Z}$. We mark the
triangles of~$\widetilde{Z}$ by elements of~$S_3$ in such a way
that the action of~$S_3$ on~$\widetilde{Z}$ coincides with the
action of~$S_3$ on itself by left shifts. Then four-dimensional
cells of the decomposition~$\widetilde{Q}=\widetilde{Z}\times
\widetilde{Z}$ are marked by ordered pairs of elements of~$S_3$
and the action of~$S_3$ is diagonal. Therefore the
decomposition~$Q=\widetilde{Q}/S_3$ contains $6$ four-dimensional
cells which can be marked by permutations~$\varkappa\in S_3$ so
that under a factorization by the action of~$S_3$ the cell with
mark~$(\varkappa_1,\varkappa_2)$ goes to the cell with
mark~$\varkappa_1^{-1}\varkappa_2$. The four-dimensional cell
of~$Q$ with mark~$\varkappa$ will be denoted by~$P_{\varkappa}$.

\begin{figure}
\unitlength=0.3cm
\begin{picture}(20,16)

\put(10,8){\begin{picture}(0,0)

\put(-6,0){\circle*{0.5}}

\put(6,0){\circle*{0.5}}

\put(-3,5){\circle*{0.5}}

\put(-3,-5){\circle*{0.5}}

\put(3,5){\circle*{0.5}}

\put(3,-5){\circle*{0.5}}

\put(-6.5,0){\line(1,0){13}}

\put(-8,0){\line(1,0){1}}

\put(7,0){\line(1,0){1}}

\put(-9.5,0){\line(1,0){1}}

\put(8.5,0){\line(1,0){1}}

\put(-3.6,6){\line(3,-5){7.2}}

\put(-3.9,6.5){\line(-3,5){0.75}}

\put(3.9,-6.5){\line(3,-5){0.75}}

\put(3.6,6){\line(-3,-5){7.2}}

\put(3.9,6.5){\line(3,5){0.75}}

\put(-3.9,-6.5){\line(-3,-5){0.75}}

\put(0.8,0.3){0}

\put(-3,5.3){1}

\put(6.3,0.3){1}

\put(-3,-6.3){1}

\put(2.3,5.3){2}

\put(-6.9,0.3){2}

\put(2.3,-6.3){2}

\put(4.5,6.3){$\ell_2$}

\put(-9,-1.3){$\ell_0$}

\put(4.7,-7.3){$\ell_1$}

\put(2.7,1.2){\large$e$}

\put(-1.1,3){$(01)$}

\put(1.9,-2){$(12)$}

\put(-4.1,-2){$(02)$}

\put(-4.5,1.2){$(012)$}

\put(-1.5,-3.8){$(021)$}

\thicklines

\put(-3,5){\line(1,0){6}}

\put(-3,-5){\line(1,0){6}}

\put(-3,5){\line(-3,-5){3}}

\put(-3,-5){\line(-3,5){3}}

\put(3,5){\line(3,-5){3}}

\put(3,-5){\line(3,5){3}}

\end{picture}}

\end{picture}

\caption{The decomposition~$\widetilde{Z}$ of the
torus~$W_{\R}/\Lambda_{\R}$ into $6$ triangles} \label{fig_6}
\end{figure}
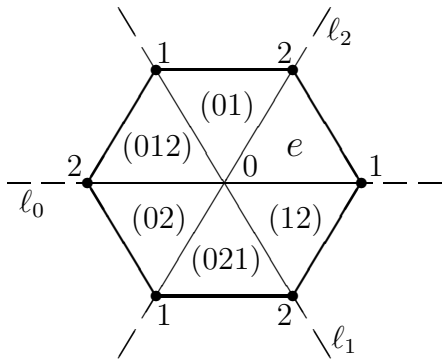

We mark vertices of the decomposition~$\widetilde{Z}$ by
elements~$0,1,2\in\Z_3$ as it is shown in Fig.~\ref{fig_6}. Then
vertices of the decomposition~$\widetilde{Q}$ are marked by
pairs~$(a_1,a_2)\in\Z_3\times\Z_3$. The vertex corresponding to a
pair~$(a_1,a_2)$ will be denoted by~$u(a_1,a_2)$. The
vertices~$u(a_1,a_2)$ are fixed by the action of~$S_3$. Hence
these vertices are exactly the vertices of the decomposition~$Q$.
Thus $Q$ is a $9$-vertex decomposition of~$\CP^2$ into $6$
four-dimensional cells~$P_{\varkappa}$. Every cell~$P_{\varkappa}$
is a prism~$\Delta^2\times\Delta^2$ and the vertices of each
multiplier~$\Delta^2$ are marked by pairwise distinct
elements~$0,1,2$. Faces of the prism~$\Delta^2\times\Delta^2$ are
in one-to-one correspondence with pairs of
subsets~$A_1,A_2\subset\Z_3$. To a pair $(A_1,A_2)$ we assign the
face spanned by all vertices~$u(a_1,a_2)$ such that $a_1\in A_1$
and $a_2\in A_2$. The face corresponding to a pair~$(A_1,A_2)$ is
called a face of type~$(A_1,A_2)$.

Now, for two distinct four-dimensional cells~$P_{\varkappa_1}$
and~$P_{\varkappa_2}$ of~$Q$ we compute the number of common
facets of the cells~$P_{\varkappa_1}$ and~$P_{\varkappa_2}$. If
one of the permutations~$\varkappa_1$ and~$\varkappa_2$ is even
and the other is odd, then the cells~$P_{\varkappa_1}$
and~$P_{\varkappa_2}$ have two common facets of
types~$(\{a_1+1,a_1+2\},\{1,2,3\})$
and~$(\{1,2,3\},\{a_2+1,a_2+2\})$,
where~$\varkappa_2=\sigma_{a_1}\varkappa_1=\varkappa_1\sigma_{a_2}$.
If the permutations~$\varkappa_1$ and~$\varkappa_2$ are either
both even or both odd, then the cells~$P_{\varkappa_1}$
and~$P_{\varkappa_2}$ have no common facets. Thus the
decomposition~$Q$ contains $9$ quadrangular two-dimensional cells
each of which is contained in exactly two three-dimensional cells
and exactly two four-dimensional cells. Perform the following
operation with the decomposition~$Q$. Delete from this
decomposition those $9$ quadrangular two-dimensional faces and
unite the pair of common facets of each two
prisms~$P_{\varkappa_1}$ and~$P_{\varkappa_2}$ so as to obtain one
three-dimensional cell with combinatorial type of the polytope
shown in Fig.~\ref{fig_prism},\textit{а}. Four-dimensional cells
of the obtained decomposition are just four-dimensional cells
of~$Q$. However the cell decompositions of their boundaries are
different. Every four-dimensional cell of the obtained
decomposition is a four-dimensional disk with boundary decomposed
into~$3$ three-dimensional cells combinatorially equivalent to the
polytope shown in Fig.~\ref{fig_prism},\textit{а}. The obtained
cell decomposition of~$\CP^2$ will be denoted by~$Q^{(1)}$ and its
four-dimensional cells  will be denoted by~$P_{\varkappa}^{(1)}$.

\begin{figure}
\unitlength=0.3cm
\begin{picture}(40,17)
\multiput(0,3)(25,0){2}{\begin{picture}(0,0)

\put(3,0){\circle*{0.3}}

\put(1,4){\circle*{0.3}}

\put(11,4){\circle*{0.3}}

\put(7,6){\circle*{0.3}}

\put(3,10){\circle*{0.3}}

\put(15,10){\circle*{0.3}}

\put(1,14){\circle*{0.3}}

\put(11,14){\circle*{0.3}}


\multiput(1,4)(1.5,0){7}{\line(1,0){1}}

\put(1,4){\line(0,1){10}}

\put(3,0){\line(0,1){10}}


\multiput(11,4)(0,1.5){7}{\line(0,1){1}}

\put(1,14){\line(1,0){10}}

\put(1,4){\line(1,-2){2}}

\put(1,14){\line(1,-2){2}}

\put(3,0){\line(2,1){8}}

\put(3,10){\line(2,1){8}}

\put(7,6){\line(2,1){8}}

\put(3,10){\line(1,-1){4}}

\put(11,14){\line(1,-1){4}}

\put(3,0){\line(2,3){4}}

\put(11,4){\line(2,3){4}}

\end{picture}}

\put(26,7){\line(1,3){2}}

\put(28,13){\line(1,0){12}}

\put(32,9){\line(2,-1){4}}

\multiput(26,17)(1.7,-1.7){6}{\line(1,-1){1.2}}

\put(6,0){\textit{a}}

\put(31,0){\textit{b}}

\end{picture}

\caption{Three-dimensional cells of the decompositions~$Q^{(1)}$
and $Q^{(2)}$} \label{fig_prism}
\end{figure}
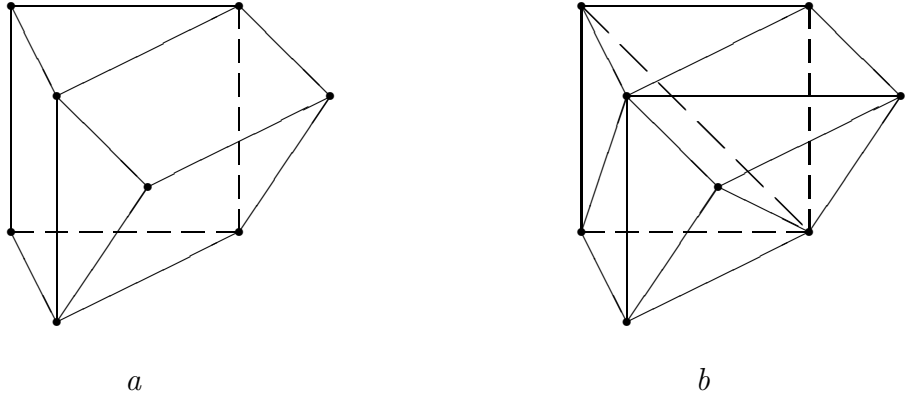

Now we consider the two-dimensional skeleton of~$Q^{(1)}$. It
consists of triangular and quadrangular cells. Every quadrangular
cell has vertices~$u(a,b)$, $u(a+1,b)$, $u(a+1,b+1)$, and
$u(a,b+1)$ for some $a,b\in\Z_3$. Decompose every such
quadrangular cell into two triangular cells by the diagonal with
vertices~$u(a,b)$ and~$u(a+1,b+1)$. We denote the obtained cell
decomposition by~$Q^{(2)}$ and denote its four-dimensional cells
by~$P_{\varkappa}^{(2)}$. Certainly, the combinatorial types of
three-dimensional and four-dimensional cells have changed under
this operation. In particular, all three-dimensional cells have
become isomorphic to the suspension over a hexagon (see
Fig.~\ref{fig_prism},\textit{b}). Notice that every two distinct
four-dimensional cells~$P_{\varkappa_1}^{(2)}$
and~$P_{\varkappa_2}^{(2)}$ of~$Q^{(2)}$ have no common facets if
the permutations~$\varkappa_1$ and~$\varkappa_2$ are either both
even or both odd and have a unique common facet if one of the
permutations~$\varkappa_1$ and~$\varkappa_2$ is even and the other
is odd. (The same assertion holds for~$Q^{(1)}$.)

The two-dimensional skeleton of the decomposition~$Q^{(2)}$ is a
decomposition into triangles. It can be immediately checked that
the two-dimensional skeleton of~$Q^{(2)}$ is a simplicial complex,
that is, it does not contain multiple edges and triangles with
coinciding boundary. (This fact will be very important for us.)
Indeed, the construction described above provides an explicit
description of two-dimensional faces of~$Q^{(2)}$. Every
two-dimensional face of~$Q^{(2)}$ is a triangle spanned by a set
of vertices of one of the following types.

1) $u(a,0)$, $u(a,1)$, $u(a,2)$.

2) $u(0,a)$, $u(1,a)$, $u(2,a)$.

3) $u(a+1,b+1)$, $u(a+2,b+1)$, $u(a+2,b+2)$.

4) $u(a+1,b+1)$, $u(a+1,b+2)$, $u(a+2,b+2)$.

Suppose that~$\varkappa_1,\varkappa_2\in S_3$ are permutations one
of which is even and the other is odd. Then the triangle of
type~1) is contained in~$P^{(2)}_{\varkappa_1}\cap
P^{(2)}_{\varkappa_2}$ if and only
if~$\varkappa_2\ne\sigma_a\varkappa_1$, the triangle of type~2) is
contained in~$P^{(2)}_{\varkappa_1}\cap P^{(2)}_{\varkappa_2}$ if
and only if~$\varkappa_2\ne\varkappa_1\sigma_a$, and the triangle
of type~3) or of type~4) is contained
in~$P^{(2)}_{\varkappa_1}\cap P^{(2)}_{\varkappa_2}$ if and only
if exactly one of the two equalities
$\varkappa_2=\sigma_a\varkappa_1$ and
$\varkappa_2=\varkappa_1\sigma_b$ holds.

Now let us construct a $15$-vertex triangulation~$Y$ of~$\CP^2$ in
the following way. Introduce a new vertex~$u(\varkappa)$ in the
interior of every  four-dimensional cell~$P^{(2)}_{\varkappa}$.
Decompose every cell~$P^{(2)}_{\varkappa}$ into the cones over its
facets with vertex~$u(\varkappa)$. Every three-dimensional
face~$F$ of~$Q^{(2)}$ is contained in two four-dimensional
cells~$P^{(2)}_{\varkappa_1}$ and~$P^{(2)}_{\varkappa_2}$. Uniting
the cones over~$F$ with vertices~$u(\varkappa_1)$
and~$u(\varkappa_2)$ we obtain the suspension over~$F$.
Triangulate this suspension as the join of the segment with
endpoints~$u(\varkappa_1)$ and~$u(\varkappa_2)$ and the
triangulation of~$\partial F$. (Recall that the two-dimensional
skeleton of~$Q^{(2)}$ is a simplicial complex.) Performing the
described operation for all three-dimensional faces of~$Q^{(2)}$,
we obtain a triangulation, which we denote by~$Y$. The
triangulation~$Y$ has $15$ vertices among which there are $9$
vertices~$u(a,b)$ and $6$ vertices~$u(\varkappa)$.
Four-dimensional simplices are spanned by the following sets of
vertices.

 1) $u(\varkappa)$, $u(\sigma_b\varkappa)$, $u(a,0)$,
$u(a,1)$, $u(a,2)$,\ \ \ $a\ne b$;

2) $u(\varkappa)$, $u(\varkappa\sigma_b)$, $u(0,a)$, $u(1,a)$,
$u(2,a)$,\ \ \ $a\ne b$;

3) $u(\varkappa)$, $u(\sigma_a\varkappa)$, $u(a+1,b+1)$,
$u(a+2,b+1)$, $u(a+2,b+2)$,\ \ \ $a\ne \varkappa(b)$;

4) $u(\varkappa)$, $u(\sigma_a\varkappa)$, $u(a+1,b+1)$,
$u(a+1,b+2)$, $u(a+2,b+2)$,\ \ \ $a\ne \varkappa(b)$;

5) $u(\varkappa)$, $u(\varkappa\sigma_b)$, $u(a+1,b+1)$,
$u(a+2,b+1)$, $u(a+2,b+2)$,\ \ \ $a\ne \varkappa(b)$;

6) $u(\varkappa)$, $u(\varkappa\sigma_b)$, $u(a+1,b+1)$,
$u(a+1,b+2)$, $u(a+2,b+2)$,\ \ \ $a\ne \varkappa(b)$.

The number of simplices of each type is equal to~$18$.

\begin{propos}
The triangulation~$Y$ is isomorphic to the simplicial complex~$X$.
\end{propos}
\begin{proof}
Arrange a one-to-one correspondence between the vertices of~$X$
and the vertices of~$Y$ in the following way.
\begin{gather*}
(12)(34)\mapsto u(\sigma_0);\quad (13)(24)\mapsto
u(\sigma_1);\quad (14)(23)\mapsto u(\sigma_2);\\
(4,b)\mapsto u\bigl((012)^b\bigr),\qquad b=1,2,3;\\
(a,b)\mapsto u(-a-b,-a+b),\qquad a,b=1,2,3,
\end{gather*}
where the sums in the last formula are taken modulo~$3$. Notice
that a set of pairwise distinct vertices of~$X$ spans a simplex if
and only if it contains no pair of vertices of the
form~$(a,b_1),(a,b_2)$ or of the form~$\nu_1,\nu_2$ and no triple
of vertices of the form~$(a_1,b),(a_2,b),(a_3,b)$ or of the
form~$\nu,(a,b),(\nu(a),b)$. A set of pairwise distinct vertices
of~$Y$ spans a simplex if and only if it contains no pair of
vertices of the form~$u(a_1,b_1),u(a_2,b_2)$ with
$a_1+b_1=a_2+b_2$ or of the form~$u(\varkappa_1),u(\varkappa_2)$
with $\varkappa_1$ and $\varkappa_2$ either both even or both odd
and no triple of vertices of one of the forms
\begin{gather*}
u(a_0,a_0),u(a_1,a_1+1),u(a_2,a_2+2);\\
u(\varkappa),u(a+1,\varkappa(a)+1),u(a+2,\varkappa(a)+2);\\
u(\varkappa),u(\varkappa\sigma_b),u(\varkappa(b),b).
\end{gather*}
To prove that the complexes~$X$ and~$Y$ are isomorphic one
suffices to notice the constructed one-to-one correspondence
between their vertices takes the ``prohibited'' pairs and triples
of vertices for the complex~$X$ to the ``prohibited'' pairs and
triples of vertices for the complex~$Y$. The latter assertion can
be checked immediately.
\end{proof}
\begin{remark}
Recall that the obtained result on the relationship of the
triangulation~$X$ with complex crystallographic groups is weaker
than a similar result of B.\,Morin and M.\,Yoshida for K\"uhnel's
triangulation~$K$. First, the author does not know whether the
triangulation~$X$ can be obtained as the quotient of a rectilinear
triangulation of~$\bC^2$ by a crystallographic group. Second, the
above construction of the triangulation~$Y$ does not explain why
the automorphism group of the triangulation~$Y$ is so big.
\end{remark}

The author is grateful to V.\,M.\,Buchstaber for constant
attention to this work and fruitful discussion and to
S.\,A.\,Melikhov for useful comments.

\end{document}